\documentclass{amsart}
\usepackage{basic-preamble}
\begin{document}



\author[Dos Santos, H.S.]{Helen Samara Dos Santos}
\address{Department of Mathematics and Statistics,
	Memorial University of Newfoundland,
St. John's, NL, A1C5S7, Canada}
\email{helensds@mun.ca}

\author[Hornhardt, C.]{Caio De Naday Hornhardt}
\address{Department of Mathematics and Statistics,
	Memorial University of Newfoundland,
St. John's, NL, A1C5S7, Canada}
\email{cdnh22@mun.ca}

\author[Kochetov, M.]{Mikhail Kochetov}
\address{Department of Mathematics and Statistics,
	Memorial University of Newfoundland,
St. John's, NL, A1C5S7, Canada}

\email{mikhail@mun.ca}


\thanks{The authors acknowledge financial support by Discovery Grant 341792-2013 of the Natural Sciences and Engineering Research Council (NSERC) of Canada and by Memorial University of Newfoundland.}


\date{}

\title[Group gradings on the superalgebras $M(m,n)$, $A(m,n)$ and $P(n)$]{Group gradings on the superalgebras\\ $M(m,n)$, $A(m,n)$ and $P(n)$}




\begin{abstract}
	We classify gradings by arbitrary abelian groups on the classical simple Lie superalgebras $P(n)$, $n \geq 2$, and on the simple associative superalgebras $M(m,n)$, $m, n \geq 1$, over an algebraically closed field: fine gradings up to equivalence and $G$-gradings, for a fixed group $G$, up to isomorphism.
	As a corollary, we also classify up to isomorphism the $G$-gradings on the classical Lie superalgebra $A(m,n)$ that are induced from $G$-gradings on $M(m+1,n+1)$. In the case of Lie superalgebras, the characteristic is assumed to be $0$.
\end{abstract}


\maketitle

\section{Introduction}

In the past two decades, gradings on Lie algebras by arbitrary abelian groups have been extensively studied. For finite-dimensional simple Lie algebras over an algebraically closed field $\FF$, the classification of fine gradings up to equivalence has recently been completed (assuming $\Char \FF = 0$) by efforts of many authors --- see the monograph \cite[Chapters 3--6]{livromicha} and the references therein, and also \cite{YuExc} and \cite{E14}. For a fixed abelian group $G$, the classification of $G$-gradings up to isomorphism is also known (assuming $\Char \FF \neq 2$), except for types $E_6$, $E_7$ and $E_8$ --- see \cite{livromicha} and \cite{EK_d4}.

This paper is devoted to gradings on finite-dimensional simple Lie superalgebras. Over an algebraically closed field of characteristic $0$, such superalgebras were classified by V.~G.~Kac in \cite{kacZ,artigokac} (see also \cite{livrosuperalgebra}). In \cite{kacZ}, there is also a classification of $\ZZ$-gradings on these superalgebras. More recently, gradings by arbitrary abelian groups have been considered. Fine gradings on the exceptional simple Lie superalgebras, namely, $D(2,1;\alpha)$, $G(3)$ and $F(4)$, were classified in \cite{artigoelduque} and all gradings on the series $Q(n)$, $n\geq 2$, were classified in \cite{paper-Qn}. A description of gradings on matrix superalgebras, here denoted by $M(m,n)$ (see Section \ref{sec:Mmn}), was given in \cite{BS}, but the isomorphism problem was left open and fine gradings were not considered.

The initial goal of this work was to classify abelian group gradings on the series $P(n)$, $n\geq 2$, and thereby complete the classification of gradings on the so-called ``strange Lie superalgebras''. Our approach led us to the study of gradings on the associative superalgebras $M(m,n)$ and the closely related Lie superalgebras $A(m,n)$.

Throughout this work, the canonical $\ZZ_2$-grading of a superalgebra will be denoted by superscripts, reserving subscripts for the components of other gradings. Thus, a $G$-grading on a superalgebra $A = A\even \oplus A\odd$ is a vector space decomposition $\Gamma:\,A = \bigoplus_{g \in G} A_g$ such that $A_g A_h\subseteq A_{gh}$, for all $g,h\in G$, and each $A_g$ is compatible with the superalgebra structure, i.e., $A_g=A_g^\bz \oplus A_g^\bo$. Note that $G$-gradings on a superalgebra can be seen as $G\times \ZZ_2$-gradings on the underlying algebra. For the superalgebras under consideration, namely, $M(m,n)$, $A(m,n)$ and $P(n)$, the canonical $\ZZ_2$-grading can be refined to a canonical $\ZZ$-grading, whose components will be denoted by superscripts $-1, 0, 1$.

Only gradings by abelian groups are discussed in this work, which is no loss of generality in the case of simple Lie superalgebras, because the support always generates an abelian group.
All the (super)algebras and vector (super)spaces are assumed to be finite-dimensional over a fixed algebraically closed field $\FF$. 
When dealing with the Lie superalgebras $A(m,n)$ and $P(n)$, we will also assume $\Char \FF = 0$.

The paper is structured as follows. Sections \ref{sec:generalities} and \ref{sec:gradings-on-matrix-algebras} have no original results. In the former, we introduce all basic definitions and a few general results for future reference, and the latter is a review of the classification of gradings on matrix algebras closely following \cite[Chapter 2]{livromicha}, with a slight change in notation. 

Section \ref{sec:Mmn} is devoted to the associative superalgebras $M(m,n)$, which
have two kinds of gradings: the \emph{even gradings} are compatible with the canonical $\ZZ$-grading and the \emph{odd gradings} are not. 
(The latter can occur only if $m=n$.) 
The classification results for even gradings are Theorems \ref{thm:even-assc-iso} ($G$-gradings up to isomorphism) 
and \ref{thm:class-fine-even} (fine gradings up to equivalence). 
We present two descriptions of odd gradings: one as $G\times \ZZ_2$-gradings on the underlying matrix algebra (see Subsection \ref{ssec:grds-on-superalgebras}) 
and the other purely in terms of the group $G$ (see Subsection \ref{ssec:second-odd}). We classify odd gradings in Theorems \ref{thm:first-odd-iso} and \ref{thm:2nd-odd-iso} ($G$-gradings up to isomorphism) and in Theorem \ref{thm:class-fine-odd} (fine gradings up to equivalence).

In Section \ref{sec:Amn}, we consider gradings on the Lie superalgebras $A(m,n)$, but only those that are induced from $M(m+1, n+1)$ (see Definition \ref{def:Type-I}). We classify them up to isomorphism in Theorem \ref{thm:even-Lie-iso} (even gradings) and in Theorem \ref{thm:first-odd-Lie-iso} and Corollary \ref{cor:2nd-odd-Lie-iso} (odd gradings).

In Section \ref{sec:Pn}, we classify gradings on the Lie superalgebras $P(n)$: see Theorem \ref{thm:Pn-iso} for $G$-gradings up to isomorphism 
and Theorem \ref{thm:class-fine-Pn} for fine gradings up to equivalence.

\section{Generalities on gradings}\label{sec:generalities}

The purpose of this section is to fix notation and recall definitions concerning graded algebras and graded modules. 

\subsection{Gradings on vector spaces and (bi)modules}\label{subsec:graded-bimodules}

Let $G$ be a group. By a \emph{$G$-grading} on a vector space $V$ we mean simply a vector space decomposition $\Gamma:\,V = \bigoplus_{g \in G} V_g$ where the summands are labeled by elements of $G$. If $\Gamma$ is fixed, $V$ is referred to as a {\em $G$-graded vector space}. A subspace $W \subseteq V$ is said to be \emph{graded} if $W = \bigoplus_{g \in G} (W \cap V_g)$. We will refer to $\ZZ_2$-graded vector spaces as \emph{superspaces} and their graded subspaces as \emph{subsuperspaces}.

An element $v$ in a graded vector space $V = \bigoplus_{g \in G} V_g$ is said to be \emph{homogeneous} if $v\in V_g$ for some $g\in G$. 
If $0\ne v\in V_g$, we will say that $g$ is the \emph{degree} of $v$ and write $\deg v = g$. 
In reference to the canonical $\ZZ_2$-grading of a superspace, we will instead speak of the \emph{parity} of $v$ and write $|v| = g$.
Every time we write $\deg v$ or $|v|$, it should be understood that $v$ is a nonzero homogeneous element.

\begin{defi}
	Given two $G$-graded vector spaces, $V=\bigoplus_{g\in G} V_g$ and $W=\bigoplus_{g\in G} W_g$, we define their tensor product to be the vector space $V\otimes W$
	together with the $G$-grading given by $(V \otimes W)_g = \bigoplus_{ab=g} V_{a} \otimes W_{b}$.
\end{defi}

The concept of grading on a vector space is connected to gradings on algebras by means of the following:

\begin{defi}
	If $V=\bigoplus_{g\in G} V_{g}$ and $W=\bigoplus_{g\in G} W_{g}$ are two graded vector spaces and $T: V\rightarrow W$ is a linear map, we say that $T$ is \emph{homogeneous of degree $t$}, for some $t\in G$, if $T(V_g)\subseteq W_{tg}$ for all $g\in G$.
\end{defi}

If $S: U\rightarrow V$ and $T: V\rightarrow W$ are homogeneous linear maps of degrees $s$ and $t$, respectively, 
then the composition $T\circ S$ is homogeneous of degree $ts$.
We define the {\em space of graded linear transformations} from $V$ to $W$ to be:
\[ \Hom^{\text{gr}} (V,W) = \bigoplus_{g\in G} \Hom (V,W)_{g}\]
%
where $\Hom (V,W)_{g}$ denotes the set of all linear maps from $V$ to $W$ that are homogeneous of degree $g$. 
If we assume $V$ to be finite-dimensional then we have $\Hom(V,W)=\Hom^{\gr}(V,W)$ and, in particular, $\End (V) = \bigoplus_{g\in G} \End (V)_g$ is a graded algebra.
We also note that $V$ becomes a graded module over $\End(V)$ in the following sense:

\begin{defi}
	Let $A$ be a $G$-graded algebra (associative or Lie) and let $V$ be a (left) module over $A$ that is also a $G$-graded vector space. We say that $V$ is a \emph{graded $A$-module} if $A_g \cdot V_h \subseteq V_{gh}$, for all $g$,$h\in G$. The concept of $G$-\emph{graded bimodule} is defined similarly.
\end{defi}

If we have a $G$-grading on a Lie superalgebra $L=L\even \oplus L\odd$ then, in particular, we have a grading on the Lie algebra $L\even$ and a grading on the space $L\odd$ that makes it a graded $L\even$-module. If we have a $G$-grading on an associative superalgebra $C=C\even \oplus C\odd$, then $C\odd$ becomes a graded bimodule over $C\even$.

If $ \Gamma$ is a $G$-grading on a vector space $V$ and $g\in G$, we denote by $\Gamma^{[g]} $ the grading given by relabeling the component 
$V_h$ as $V_{hg}$, for all $h \in G$. This is called the \emph{(right) shift of the grading $\Gamma$ by $g$}. 
We denote the graded space $(V, \,  \Gamma^{[g]})$ by $V^{[g]}$.

From now on, we assume that $G$ is abelian.
If $V$ is a graded module over a graded algebra (or a graded bimodule over a pair of graded algebras), then $V^{[g]}$ is also a graded (bi)module. 
We will make use of the following partial converse (see e.g. \cite[Proposition 3.5]{paper-Qn}):

\begin{lemma}\label{lemma:simplebimodule}
	Let $A$ and $B$ be $G$-graded algebras and let $V$ be a finite-dimensional (ungraded) simple $A$-module or $(A,B)$-bimodule.  If $\Gamma$ and $\Gamma'$ are  two $G$-gradings that make $V$ a graded (bi)module, then $\Gamma'$ is a  shift of $\Gamma$.\qed
\end{lemma}

Certain shifts of grading may be applied to graded $\ZZ$- or $\ZZ_2$-superalgebras. In the case of a $\ZZ$-superalgebra $L=L^{-1}\oplus L^{0}\oplus L^{1}$, we have the following:

\begin{lemma}\label{lemma:opposite-directions}
	Let $L=L^{-1}\oplus L^0\oplus L^1$ be a $\ZZ$-superalgebra such that $L^1\, L^{-1}\neq 0$. If we shift the grading on $L^1$ by $g\in G$ and the grading on $L^{-1}$ by $g' \in G$, then we have a grading on $L$ if and only if $g' = g^{-1}$. \qed
\end{lemma}

We will describe this situation as \emph{shift in opposite directions}.

\subsection{Universal grading group, equivalence and isomorphism of gradings}

There is a concept of grading not involving groups. A \emph{set grading} on a (super)algebra $A$ is a decomposition $\Gamma:\,A=\bigoplus_{s\in S}A_s$ as a direct sum of sub\-(su\-per)\-spa\-ces indexed by a set $S$ and having the property that, for any $s_1,s_2\in S$ with $A_{s_1}A_{s_2}\ne 0$, there exists $s_3\in S$ such that $A_{s_1}A_{s_2}\subseteq A_{s_3}$. The \emph{support} of $\Gamma$ (or of $A$) is defined to be the set $\supp(\Gamma) := \{s\in S \mid A_s \neq 0\}$.
Similarly, $\supp_\bz(\Gamma) := \{s\in S \mid A_s^\bz \neq 0\}$ and $\supp_\bo(\Gamma) := \{s\in S \mid A_s^\bo \neq 0\}$.

For a set grading $\Gamma:\,A=\bigoplus_{s\in S}A_s$, there may or may not exist a group $G$ containing $\supp(\Gamma)$ that makes $\Gamma$ a $G$-grading. 
If such a group exists, $\Gamma$ is said to be a {\em group grading}. (As already mentioned, we only consider abelian group gradings in this paper.) 
However, $G$ is usually not unique even if we require that it should be generated by $\supp(\Gamma)$. 
The {\em universal (abelian) grading group} of $\Gamma$ is generated by $\supp(\Gamma)$ and has the defining relations 
$s_1s_2=s_3$ for all $s_1,s_2,s_3\in S$ such that $0\neq A_{s_1}A_{s_2}\subseteq A_{s_3}$. 
This group is universal among all (abelian) groups that realize the grading $\Gamma$ (see e.g. \cite[Chapter 1]{livromicha} for details).

Let $\Gamma:\,A=\bigoplus_{g\in G} A_g$ and $\Delta:\,B=\bigoplus_{h\in H} B_h$ be two group gradings on the (super)algebras $A$ and $B$, with supports $S$ and $T$, respectively.
We say that $\Gamma$ and $\Delta$ are {\em equivalent} if there exists an isomorphism of (super)algebras $\vphi: A\to B$ and a bijection $\alpha: S\to T$ such that $\vphi(A_s)=B_{\alpha(s)}$ for all $s\in S$. If $G$ and $H$ are universal grading groups then $\alpha$ extends to an isomorphism $G\to H$. In the case $G=H$, the $G$-gradings $\Gamma$ and $\Delta$ are {\em isomorphic} if $A$ and $B$ are isomorphic as $G$-graded (super)algebras, i.e., if there exists an isomorphism of (super)algebras $\vphi: A\to B$ such that $\vphi(A_g)=B_g$ for all $g\in G$.

If $\Gamma:\,A=\bigoplus_{g\in G} A_g$ and $\Gamma':\,A=\bigoplus_{h\in H} A'_h$ are two gradings on the same (super)algebra $A$, with supports $S$ and $T$, respectively, then we will say that $\Gamma'$ is a {\em refinement} of $\Gamma$ (or $\Gamma$ is a {\em coarsening} of $\Gamma'$) if, for any $t\in T$, there exists (unique) $s\in S$ such that $A'_t\subseteq A_s$. If, moreover, $A'_t\ne A_s$ for at least one $t\in T$, then the refinement is said to be {\em proper}. A grading $\Gamma$ is said to be {\em fine} if it does not admit any proper refinements. 
Note that if $A$ is a superalgebra then $A=\bigoplus_{(g,i)\in G\times\mathbb{Z}_2}A_g^i$ is a refinement of $\Gamma$. 
It follows that if $\Gamma$ is fine then the sets $\supp_\bz(\Gamma)$ and $\supp_\bo(\Gamma)$ are disjoint. 
If, moreover, $G$ is the universal group of $\Gamma$, then the superalgebra structure on $A$ is given by the unique homomorphism $p: G \to \ZZ_2$ 
that sends $\supp_\bz(\Gamma)$ to $\bar 0$ and $\supp_\bo(\Gamma)$ to $\bar 1$.

\begin{defi}
	Let $G$ and $H$ be groups, $\alpha:G\to H$ be a group homomorphism and $\Gamma:\,A=\bigoplus_{g\in G} A_g$ be a $G$-grading. The \emph{coarsening of $\Gamma$ induced by $\alpha$} is the $H$-grading ${}^\alpha \Gamma: A= \bigoplus_{h\in H} B_h$ where
	$ B_h = \bigoplus_{g\in \alpha\inv (h)} A_g$. (This coarsening is not necessarily proper.)
\end{defi}

The following result appears to be ``folklore''. We include a proof for completeness.

\begin{lemma}\label{lemma:universal-grp}
	Let $\mathcal{F}=\{\Gamma_i\}_{i\in I}$, be a family of pairwise nonequivalent fine (abelian) group gradings on a (super)algebra $A$, where $\Gamma_i$ is a $G_i$-grading and $G_i$ is generated by $\supp(\Gamma_i)$. Suppose that $\mathcal{F}$ has the following property: 
	for any grading $\Gamma$ on $A$ by an (abelian) group $H$, there exists $i\in I$ and a homomorphism $\alpha:G_i\to H$ such that $\Gamma$ 
	is isomorphic to ${}^\alpha\Gamma_i$. Then
	\begin{enumerate}[(i)]
		\item every fine (abelian) group grading on $A$ is equivalent to a unique $\Gamma_i$;
		\item for all $i$, $G_i$ is the universal (abelian) group of $\Gamma_i$.
	\end{enumerate}
\end{lemma}

\begin{proof}
	Let $\Gamma$ be a fine grading on $A$, realized over its universal group $H$. Then there is $i\in I$ and $\alpha: G_i \to H$ such that ${}^\alpha \Gamma_i \iso \Gamma$. Writing $\Gamma_i: A = \bigoplus_{g\in G_i} A_g$ and $\Gamma: A = \bigoplus_{h\in H} B_h$, we then have $\vphi \in \Aut(A)$ such that
	\[
		\vphi\,\big( \bigoplus_{g\in \alpha\inv (h)} A_g \big) = B_h
	\]
	for all $h\in H$. Since $\Gamma$ is fine, we must have $B_h \neq 0$ if, and only if, there is a unique $g\in G_i$ such that $\alpha(g) = h$, $A_g\neq 0$ and $\vphi(A_g) = B_h$. Equivalently, $\alpha$ restricts to a bijection $\supp(\Gamma_i) \to \supp(\Gamma)$ and $\vphi(A_g) = B_{\alpha(g)}$ for all $g \in S_i:= \supp (\Gamma_i)$. This proves assertion $(i)$.

	Let $G$ be the universal group of $\Gamma_i$. It follows that, for all $s_1, s_2, s_3 \in S_i$,
	\begin{equation*} \label{eq:relations-unvrsl-grp}
		\begin{split}
			& s_1s_2 = s_3 \text{ is a defining relation of } G \\
									 \iff & 0 \neq A_{s_1} A_{s_2} \subseteq A_{s_3}\\
									 \iff & 0 \neq B_{\alpha(s_1)} B_{\alpha(s_2)} \subseteq B_{\alpha (s_3)}\\
									 \iff & \alpha(s_1)\alpha(s_2) = \alpha(s_3) \text{ is a defining relation of } H.
		\end{split}
	\end{equation*}
	Therefore, the bijection $\alpha\restriction_{S_i}$ extends uniquely to an isomorphism $\beta: G\rightarrow H$.

	By the universal property of $G$, there is a unique homomorphism $\gamma: G\to G_i$ that restricts to the identity on $S_i$. Hence, the following diagram commutes:
	\begin{center}
		\begin{tikzcd}
			G \arrow[to=Gi, "\gamma"] \arrow[to = H, "\beta"]&&\\
			&& |[alias=H]|H\\
			|[alias=Gi]|G_i \arrow[to=H, "\alpha"]&&
		\end{tikzcd}
	\end{center}

	Since $\beta$ is an isomorphism, $\gamma$ must be injective. But $\gamma$ is also surjective since $S_i$ generates $G_i$. Hence $G_i$ is isomorphic to $G$. Since $\Gamma$ was an arbitrary fine grading, for each given $j\in I$, we can take $\Gamma = \Gamma_j$ (hence, $i=j$ and $H=G$). This concludes the proof of $(ii)$.
\end{proof}

\begin{defi}
	Let $\Gamma$ be a grading on an algebra $A$. We define $\Aut(\Gamma)$ as the group of all self-equivalences of $\Gamma$, i.e., automorphisms of $A$ that 
	permute the components of $\Gamma$. Let $\operatorname{Stab}(\Gamma)$ be the subgroup of $\Aut(\Gamma)$ consisting of the automorphisms that fix 
	each component of $\Gamma$. Clearly, $\operatorname{Stab}(\Gamma)$ is a normal subgroup of $\Aut(\Gamma)$, so we can define the \emph{Weil group} of 
	$\Gamma$ by $\operatorname W (\Gamma) := \frac{\Aut(\Gamma)}{\operatorname{Stab}(\Gamma)}$. The group $\operatorname W (\Gamma)$ can be seen as a subgroup
	of the permutation group of the support and also as a subgroup of the automorphism group of the universal group of $\Gamma$.
\end{defi}

\subsection{Correspondence between $G$-gradings and $\widehat G$-actions}\label{ssec:G-hat-action}

One of the most important tools for dealing with gradings by abelian groups on (super)algebras is to translate a $G$-grading into a $\widehat G$-action, where $\widehat G$ is the algebraic group of characters of $G$, \ie, group homomorphisms $G \rightarrow \FF^{\times}$. The group $\widehat{G}$ acts on any $G$-graded (super)algebra $A = \bigoplus_{g\in G} A_g$ by $\chi \cdot a = \chi(g) a$ for all $a\in A_g$ (extended to arbitrary $a\in A$ by linearity). The map given by the action of a character $\chi \in \widehat{G}$ is an automorphism of $A$. If $\FF$ is algebraically closed and $\Char \FF = 0$, then $A_g = \{ a\in A \mid \chi \cdot a = \chi (g) a\}$, so the grading can be recovered from the action.

For example, if $A=A\even \oplus A\odd$ is a superalgebra, the action of the nontrivial character of $\ZZ_2$ yields the \emph{parity automorphism} $\upsilon$, which acts as the identity on $A\even$ and as the negative identity on $A\odd$. If $A$ is a $\ZZ$-graded algebra, we get a representation $\widehat \ZZ = \FF^\times \rightarrow \Aut (A)$ given by $\lambda \mapsto \upsilon_\lambda$ where $\upsilon_{\lambda}$ acts as $\lambda^i \id$ on $A^i$.

A grading on a (super)algebra over an algebraically closed field of characteristic $0$ is said to be \emph{inner} if it corresponds to an action by inner automorphisms. For example, the inner gradings on $\Sl(n)$ (also known as Type I gradings) are precisely the restrictions of gradings on the associative algebra $M_n(\FF)$.

\section{Gradings on matrix algebras}
\label{sec:gradings-on-matrix-algebras}

In this section we will recall the classification of gradings on matrix algebras. We will follow \cite[Chapter 2]{livromicha} but use slightly different notation, which we will extend to superalgebras in Section \ref{sec:Mmn}.

The following is the graded version of a classical result (see e.g. \cite[Theorem 2.6]{livromicha}). 
We recall that a \emph{graded division algebra} is a graded unital associative algebra such that every nonzero homogeneous element is invertible.

\begin{thm}\label{thm:End-over-D}
	Let $G$ be a group and let $R$ be a $G$-graded associative algebra that has no nontrivial graded ideals and satisfies the descending chain condition on 
	graded left ideals. Then there is a $G$-graded division algebra $\D$ and a graded (right) $\D$-module $\mc{V}$ such that $R \simeq \End_{\D} (\mc{V})$ as graded algebras.\qed
\end{thm}

We apply this result to the algebra $R=M_n(\FF)$ equipped with a grading by an abelian group $G$. We will now introduce the parameters that determine 
$\mc D$ and $\mc V$, and give an explicit isomorphism $\End_{\D} (\mc{V})\simeq M_n(\FF)$ (see Definition \ref{def:explicit-grd-assoc}).

Let $\D$ be a finite-dimensional $G$-graded division algebra. It is easy to see that $T= \supp \D$ is a finite subgroup of $G$. 
Also, since we are over an algebraically closed field, each homogeneous component $\D_t$, for $t\in T$, is one-dimensional. We can choose a generator $X_t$ for each $\D_t$. It follows that, for every $u,v\in T$, there is a unique nonzero scalar $\beta (u,v)$ such that $X_u X_v = \beta (u,v) X_v X_u$. Clearly, $\beta (u,v)$ does not depend on the choice of $X_u$ and $X_v$.
The map $\beta: T\times T \rightarrow \FF^{\times}$ is a \emph{bicharacter}, \ie, both maps $\beta(t,\cdot)$ and $\beta(\cdot,t)$ are characters for every $t \in T$. It is also \emph{alternating} in the sense that $\beta (t,t) = 1$ for all $t\in T$. We define the \emph{radical} of $\beta$ as the set $\rad \beta = \{ t\in T \mid \beta(t, T) = 1 \}$. In the case we are interested in, where $\D$ is simple as an algebra, the bicharacter $\beta$ is \emph{nondegenerate}, \ie, $\rad \beta = \{e\} $. The isomorphism classes of $G$-graded division algebras that are finite-dimensional and simple as algebras are in one-to-one correspondence with the pairs $(T,\beta)$ where $T$ is a finite subgroup of $G$ and $\beta$ is an alternating nondegenerate bicharacter on $T$ (see e.g. \cite[Section 2.2]{livromicha} for a proof).

Using that the bicharacter $\beta$ is nondegenerate, we can decompose the group $T$ as $A\times B$, where the restrictions of $\beta$ to each of the subgroups $A$ and $B$ is trivial, and hence $A$ and $B$ are in duality by $\beta$. We can choose the elements $X_t\in \D_t$ in a convenient way (see \cite[Remark 2.16]{livromicha} and 
\cite[Remark 18]{EK15}) such that $X_{ab}=X_aX_b$ for all $a\in A$ and $b\in B$. Using this choice, we can define an action of $\D$ on the vector space underlying the group algebra $\FF B$, by declaring $X_a\cdot e_{b'} = \beta(a, b') e_{b'}$ and $X_b\cdot e_{b'} = e_{bb'}$. 
This action allows us to identify $\D$ with $\End{(\FF B)}$. Using the basis $\{e_{b}\mid b\in B\}$ in $\FF B$, we can see it as a matrix algebra, where
\[X_{ab}= \sum_{b'\in B} \beta(a, bb') E_{bb', b'}\]
and $E_{b'', b'}$ with $b'$, $b'' \in B$, is a matrix unit, namely, the matrix of the operator that sends $e_{b'}$ to $e_{b''}$ and sends all other basis elements to zero.

\begin{defi}
	We will refer to these matrix models of $\mc D$ as its \emph{standard realizations}.
\end{defi}

\begin{remark}\label{rmk:2-grp-transp}
	The matrix transposition is always an involution of the algebra structure. As to the grading, we have
	\[
	X_{ab}\transp = \sum_{b'\in B} \beta(a, bb') E_{b',bb'}
	   = \beta(a,b) \sum_{b''\in B} \beta(a, b^{-1}b'') E_{b^{-1}b'', b''} = \beta(a,b) X_{ab^{-1}}.
	\]
	It follows that if $T$ is an elementary 2-group, then the transposition preserves the degree. 
	In this case, we will use it to fix an identification between the graded algebras $\D$ and $\D\op$.
\end{remark}


Graded modules over a graded division algebra $\mc D$ behave similarly to vector spaces. The usual proof that every vector space has a basis, with obvious modifications, shows that every graded $\mc D$-module has a \emph{homogeneous basis}, \ie, a basis formed by homogeneous elements.
Let $\mc V$ be such a module of finite rank $k$, fix a homogeneous basis $\mc B = \{v_1, \ldots, v_k\}$ and let $g_i := \operatorname{deg} v_i$. We then have $\mc{V}\iso \ \D^{[g_1]}\oplus\cdots\oplus\D^{[g_k]}$, so, the graded $\mc D$-module $\mc V$ is determined by the $k$-tuple $\gamma = (g_1,\ldots, g_k)$. The tuple $\gamma$ is not unique. To capture the precise information that determines the isomorphism class of $\mc V$, we use the concept of \emph{multiset}, \ie, a set together with a map from it to the set of positive integers. If $\gamma = (g_1,\ldots, g_k)$ and $T=\supp \D$, we denote by $\Xi(\gamma)$ the multiset whose underlying set is $\{g_1 T,\ldots, g_k T\} \subseteq G/T$ and the multiplicity of $g_i T$, for $1\leq i\leq k$, is the number of entries of $\gamma$ that are congruent to $g_i$ modulo $T$.

Using $\mc B$ to represent the linear maps by matrices in $M_k(\D) = M_k(\FF)\tensor \D$, we now construct an explicit matrix model for $\End_{\D}(\mc V)$.

\begin{defi}\label{def:explicit-grd-assoc}
	Let $T \subseteq G$ be a finite subgroup, $\beta$ a nondegenerate alternating bicharacter on $T$, and $\gamma = (g_1, \ldots, g_k)$ a $k$-tuple of elements of $G$. Let $\D$ be a standard realization of a graded division algebra associated to $(T, \beta)$. Identify $M_k(\FF)\tensor \D \iso M_n(\FF)$ by means of the Kronecker product, where $n=k\sqrt{|T|}$. We will denote by $\Gamma(T, \beta, \gamma)$ the grading on $M_n(\FF)$ given by $\deg (E_{ij} \tensor d) := g_i (\deg d) g_j\inv$ for $i,j\in \{1, \ldots , k\}$ and homogeneous $d\in \D$, where $E_{ij}$ is the $(i,j)$-th matrix unit.
\end{defi}

If $\End(V)$, equipped with a grading, is isomorphic to $M_n(\FF)$ with $\Gamma(T, \beta, \gamma)$, we may abuse notation and also denote the grading on $\End(V)$ by $\Gamma(T,\beta,\gamma)$.
We restate \cite[Theorem 2.27]{livromicha} using our notation:

\begin{thm}\label{thm:classification-matrix}
	Two gradings, $\Gamma(T,\beta,\gamma)$ and $\Gamma(T',\beta',\gamma')$, on the algebra $M_n(\FF)$ are isomorphic if, and only if, $T=T'$, $\beta=\beta'$ and there is an element $g\in G$ such that $g \Xi(\gamma)=\Xi(\gamma')$.\qed
\end{thm}

The proof of this theorem is based on the following result (see Theorem 2.10 and Proposition 2.18 from \cite{livromicha}), which will also be needed:

\begin{prop}\label{prop:inner-automorphism}
	If $\phi: \End_\D (\mc V) \rightarrow \End_\D (\mc V')$ is an isomorphism of graded algebras, then there is a homogeneous invertible 
	$\D$-linear map $\psi: \mc V\rightarrow \mc V'$ such that $\phi(r)=\psi \circ r \circ \psi\inv$, for all $r\in \End_\D (\mc V)$.\qed
\end{prop}

\section{Gradings on $M(m,n)$}\label{sec:Mmn}

\subsection{The associative superalgebra $M(m,n)$}\label{M(m,n)}
Let $U = U\even \oplus U\odd$  be a superspace.
The algebra of endomorphisms of $U$ has an induced $\Zmod2$-grading, so it can be regarded as a superalgebra. It is convenient to write it in matrix form:
\begin{equation}\label{eq:End_U}
	\End(U) = \left(\begin{matrix}
	\End(U\even)       &  \Hom(U\odd,U\even)\\
	\Hom(U\even,U\odd) &  \End(U\odd)\\
	\end{matrix}
	\right).
\end{equation}
Choosing bases, we may assume that $U\even=\FF^m$ and $U\odd=\FF^n$, so the superalgebra $\End(U)$ can be seen as a matrix superalgebra, which is denoted by $M(m,n)$.

We may also regard $U$ as a $\ZZ$-graded vector space, putting $U^0=U\even$ and $U^1=U\odd$. By doing so, we obtain an induced $\ZZ$-grading on $M(m,n) = \End (U)$ such that
\[(\End\, U)\even =(\End\, U)^0 =
	\left(\begin{matrix}
	\End(U\even)       &  0\\
	0                  &  \End(U\odd)\\
	\end{matrix}
	\right)
\]
and $(\End\, U)\odd = (\End\, U)^{-1}\oplus (\End\, U)^1$ where
\[(\End U)^{1}=
	\left(\begin{matrix}
	0                  &  0\\
	\Hom(U\even,U\odd) &  0\\
	\end{matrix}
	\right) \,\text{ and }\, (\End U)^{-1}=
	\left(\begin{matrix}
	0                  &  \Hom(U\odd,U\even)\\
	0				   &  0                 \\
	\end{matrix}
	\right).
\]
This grading will be called the \emph{canonical $\ZZ$-grading} on $M(m,n)$.

\subsection{Automorphisms of $M(m,n)$}

It is known that the automorphisms of the superalgebra $\End(U)$ are conjugations by invertible homogeneous operators. 
(This follows, for example, from Proposition \ref{prop:inner-automorphism}.) The invertible even operators are of the form $\left( \begin{matrix}
a&0\\
0&d\\
\end{matrix}\right)$ where $a\in \GL(m)$ and $d\in \GL(n)$. The corresponding inner automorphisms of $M(m,n)$ will be called \emph{even automorphisms}. 
They form a normal subgroup of $\Aut(M(m,n))$, which we denote by $\mc E$.

The inner automorphisms given by odd operators will be called \emph{odd automorphisms}. 
Note that an invertible odd operator must be of the form $\left( \begin{matrix}
0&b\\
c&0\\
\end{matrix}\right)$ where both $b$ and $c$ are invertible, and this forces $m=n$. 
In this case, the set of odd automorphisms is a coset of $\mc E$, namely, $\pi \mc E$, 
where $\pi$ is the conjugation by the matrix $\left( \begin{matrix}
0_n & I_n\\
I_n & 0_n\\
\end{matrix}\right)$. This automorphism is called the \emph{parity transpose} and is usually denoted by superscript: 
\begin{equation*} 
\left( \begin{matrix}
a&b\\
c&d\\
\end{matrix}\right)^\pi = \left( \begin{matrix}
d&c\\
b&a\\
\end{matrix}\right).
\end{equation*} 
Thus, $\Aut (M(m,n)) = \mc E$ if $m\neq n$, and $\Aut (M(n,n)) = \mc E \rtimes \langle \pi \rangle$.

\begin{remark}\label{rmk:Aut-ZZ-superalgebra}
	It is worth noting that $\mc E$ is the automorphism group of the $\ZZ$-\-su\-per\-al\-gebra structure of $M(m,n)$, regardless of the values $m$ and $n$. Indeed, the elements of this group are conjugations by homogeneous matrices with respect to the canonical $\ZZ$-grading, but all the matrices of degree $-1$ or $1$ are degenerate.
\end{remark}

\subsection{Gradings on matrix superalgebras}\label{ssec:grds-on-superalgebras}

We are now going to generalize the results of Section \ref{sec:gradings-on-matrix-algebras} to the superalgebra $\M(m,n)$. It is clear that a $G$-graded associative superalgebra is equivalent to a $(G \times \ZZ_2)$-graded associative algebra, hence one could think that there is no new problem. But the description of gradings on matrix algebras presented in Section \ref{sec:gradings-on-matrix-algebras} does not allow us to readily see the gradings on the even and odd components of the superalgebra, so we are going to refine that description. We will denote the group $G\times \ZZ_2$ by $G^\#$ and the projection on the second factor by $p\colon G^\# \rightarrow \ZZ_2$. Also, we will abuse notation and identify $G$ with $G\times \{\barr 0\} \subseteq G^\#$.

\begin{remark}
	If the canonical $\ZZ_2$-grading is a coarsening of the $G$-grading by means of a homomorphism $p\colon G\rightarrow \ZZ_2$ (referred to as the \emph{parity homomorphism}), then we have another isomorphic copy of $G$ in $G^\#$, namely, the image of the embedding $g\mapsto (g, p (g))$, which contains the support of the $G^\#$-grading. In this case, we do not need $G^\#$ and can work with the original $G$-grading.
\end{remark}

A $G$-graded superalgebra $\mathcal D$ is called a \emph{graded division superalgebra} if every nonzero homogeneous element in $\D\even \cup \D\odd$ is invertible --- in other words, $\D$ is a $G^\#$-graded division algebra.

We separate the gradings on $\M(m,n)$ in two classes depending on the superalgebra structure on $\D$: if $\D\odd = 0$, we say that we have an \emph{even grading} and, if $\D\odd \ne 0$, we have an \emph{odd grading}. 

To see the difference between even and odd gradings, consider the $G^\#$-graded algebra $E=\End_\D (\mc U)$, where $\D$ is a $G^{\#}$-graded division algebra and $\mc U$ is a graded module over $\D$. Define
\[
	\mc U\even = \bigoplus_{g\in G^\#} \{u\in \mc U_g \mid p(g)=\barr 0\}\,\, \text{and} \,\,\mc U\odd = \bigoplus_{g\in G^\#} \{u\in \mc U_g \mid p(g)=\barr 1\}.
\]
Then $\mc U\even$ and $\mc U\odd$ are $\D\even$-modules, but they are $\D$-modules if and only if $\D\odd=0$. So, in the case of an even grading, $\mc U$ is as a direct sum of $\D$-modules, and all the information related to the canonical $\ZZ_2$-grading on $\End_\D (\mc U)$ comes from the decomposition $\mc U=\mc U\even \oplus \mc U\odd$.

\begin{defi}\label{def:even-grd-on-Mmn}
	Similarly to Definition \ref{def:explicit-grd-assoc}, we will parametrize the even gradings on $M(m,n)$ as $\Gamma(T,\beta, \gamma_0, \gamma_1)$, where the pair $(T,\beta)$ characterizes $\D$ and $\gamma_0$ and $\gamma_1$ are tuples of elements of $G$ corresponding to the degrees of homogeneous bases for $\mc U\even$ and $\mc U\odd$, respectively. Here $\gamma_0$ is a $k_0$-tuple and $\gamma_1$ is a $k_1$-tuple, with $k_0\sqrt{|T|}=m$ and $k_1\sqrt{|T|}=n$.
\end{defi}

On the other hand, in the case of an odd grading, the information about the canonical $\ZZ_2$-grading is encoded in $\D$. To see that, take a homogeneous $\D$-basis of $\mc U$ and multiply all the odd elements by some nonzero homogeneous element in $\D\odd$. This way we get a homogeneous $\D$-basis of $\mc U$ such that the degrees are all in the subgroup $G$ of $G^\#$. If we denote the $\FF$-span of this new basis by $\widetilde U$, then $E\iso \End (\widetilde U)\tensor \D$ where the first factor has the trivial $\ZZ_2$-grading.

\begin{defi}\label{def:odd-grd-on-Mmn-1}
	We parametrize the odd gradings by $\Gamma(T, \beta, \gamma)$ where $T\subseteq G^\#$ but $T\subsetneq G$, the pair $(T,\beta)$ characterizes $\D$, and $\gamma$ is a tuple of elements of $G = G\times \{\bar 0\}$ corresponding to the degrees of a homogeneous basis of $\mc U$ with only even elements.
\end{defi}

Clearly, it is impossible for an even grading to be isomorphic to an odd grading. The classification of even gradings is the following:

\begin{thm}\label{thm:even-assc-iso}
	Every even $G$-grading on the superalgebra $M(m,n)$ is isomorphic to some $\Gamma(T,\beta, \gamma_0, \gamma_1)$ as in Definition \ref{def:even-grd-on-Mmn}.
	Two even gradings, $\Gamma = \Gamma(T,\beta, \gamma_0, \gamma_1)$ and $\Gamma' = \Gamma(T',\beta', \gamma_0', \gamma_1')$, 
	are isomorphic if, and only if, $T=T'$, $\beta=\beta'$, and there is $g\in G$ such that
	\begin{enumerate}[(i)]
		\item for $m\neq n$: $g \Xi(\gamma_0)=\Xi(\gamma_0')$ and $g \Xi(\gamma_1)=\Xi(\gamma_1')$;

		\item for $m = n$: either $g \Xi(\gamma_0)=\Xi(\gamma_0')$ and $g \Xi(\gamma_1)=\Xi(\gamma_1')$ or $g\Xi(\gamma_0)=\Xi(\gamma_1')$ and $g \Xi(\gamma_1)=\Xi(\gamma_0')$.
	\end{enumerate}
\end{thm}

\begin{proof}
	We have already proved the first assertion. For the second assertion, we
	consider $\Gamma$ and $\Gamma'$ as $G^\#$-gradings on the algebra  $M(m+n)$ and use Theorem \ref{thm:classification-matrix} to conclude that they are isomorphic if, and only if, $T=T'$, $\beta=\beta'$ and there is $(g,s)\in G^\#$ such that $(g,s)\Xi(\gamma)=\Xi(\gamma')$, where $\gamma$ is the concatenation of $\gamma_0$ and $\gamma_1$, where we regard the entries as elements of $G^\# = G\times \ZZ_2$ appending $\barr{0}$ in the second coordinate of the entries of $\gamma_0$ and $\bar 1$ in the second coordinates of the entries of $\gamma_1$.

	If $m\neq n$, the condition $(g,s)\Xi(\gamma)=\Xi(\gamma)$ must have $s=\barr0$, since the size of $\gamma_0$ is different from the size of $\gamma_1$.

	If $m=n$, the condition $(g,s)\Xi(\gamma)=\Xi(\gamma')$ becomes $g \Xi(\gamma_1)=\Xi(\gamma_1')$ if $s=\barr0$ and $g \Xi(\gamma_1)=\Xi(\gamma_0')$ if $s=\barr1$.
\end{proof}

We now turn to the classification of odd gradings. Recall that here we choose the tuple $\gamma$ to consist of elements of $G$. The corresponding multiset $\Xi(\gamma)$ is contained in $\frac{G^\#}{T} \iso \frac{G}{T \cap G}$.

\begin{thm}\label{thm:first-odd-iso}
	Every odd $G$-grading on the superalgebra $M(m,n)$ is isomorphic to some $\Gamma(T,\beta, \gamma)$ as in Definition \ref{def:odd-grd-on-Mmn-1}.
	Two odd gradings, $\Gamma = \Gamma(T,\beta, \gamma)$ and $\Gamma' = \Gamma(T',\beta', \gamma')$,  
	are isomorphic if, and only if, $T=T'$, $\beta=\beta'$, and there is $g\in G$ such that $g \Xi(\gamma)=\Xi(\gamma')$.
\end{thm}

\begin{proof}
	We have already proved the first assertion. For the second assertion, 
	we again consider $\Gamma$ and $\Gamma'$ as $G^\#$-gradings and use Theorem \ref{thm:classification-matrix}: they are isomorphic if, and only if, $T=T'$, $\beta=\beta'$ and there is $(g,s)\in G^\#$ such that $(g,s)\Xi(\gamma)=\Xi(\gamma')$. 
	Since $T$ contains an element $t_1$ with $p(t_1) = \barr 1$, we may assume $s=\barr 0$.
\end{proof}

In Subsection \ref{subsec:odd-gradings}, we will show that odd gradings can exist only if $m=n$. It may be desirable to express the classification in terms of $G$ rather than $G^\#$ (as we did for even gradings). We will return to this in Subsection \ref{ssec:second-odd}.

\subsection{Even gradings and Morita context.}\label{subsec:even-gradings}

First we observe that every grading on $M(m,n)$ compatible with the $\ZZ$-superalgebra structure is an even grading. This follows from the fact that $T=\supp \D$ is a finite group, and if a finite group is contained in $G\times \ZZ$, then it must be contained in $G\times \{0\}$. Hence, when we look at the corresponding $(G\times\ZZ_2)$-grading, we have that $T\subseteq G$, so no element of $\D$ has an odd degree.

The converse is also true. Actually, we can prove a stronger assertion: if we write $\M(m,n)$ as in Equation \eqref{eq:End_U}, the subspaces given by each of the four blocks are graded. To capture this information, it is convenient to use the concepts of Morita context and Morita algebra.

Recall that a \emph{Morita context} is a sextuple $\mathcal{C} = (R, S, M, N, \vphi, \psi )$ where $R$ and $S$ are unital associative algebras, $M$ is an $(R,S)$-bimodule, $N$ is a $(S,R)$-bimodule and $\vphi: M\tensor_{S} N\rightarrow R$ and $\psi: N\tensor_{R} M\rightarrow S$ are bilinear maps satisfying the necessary and sufficient conditions for
\begin{equation*}
	C = \left(\begin{matrix}\label{eq:morita-algebra}
		R   &  M\\
		N   &  S\\
	\end{matrix}
	\right)
\end{equation*}
to be an associative algebra, \ie,
\[\vphi(m_1\tensor n_1)\cdot m_2 = m_1\cdot \psi(n_1\tensor m_2) \text{ and }\psi(n_1\tensor m_1)\cdot n_2 = n_1\cdot \vphi(m_1\tensor n_2)\]
\noindent for all $m_1,m_2\in M$ and $n_1,n_2\in N$.

We can associate a Morita context to a superspace $U = U\even \oplus U\odd$ by taking $R = \End(U\even)$, $S = \End(U\odd)$, $M = \Hom(U\odd, U\even)$, 
$N = \Hom (U\even, U\odd)$, with $\vphi$ and $\psi$ given by composition of operators.

Given an algebra $C$ as above and the idempotent 
$
\epsilon = \left(\begin{matrix}
1  &  0\\
0   &  0\\
\end{matrix}
\right)
$, 
we can recover all the data of the Morita context (up to isomorphism): $R \iso \epsilon C \epsilon$, $S \iso (1 - \epsilon) C (1 - \epsilon)$, $M \iso \epsilon C (1-\epsilon)$, $N \iso (1-\epsilon) C \epsilon$ and $\phi$ and $\psi$ are given by multiplication in $C$. In other words, the concept of Morita context is equivalent to the concept of \emph{Morita algebra}, which is a pair $(C,\epsilon)$ where $C$ is a unital associative algebra and $\epsilon\in C$ is an idempotent. 
For example, we may consider $\M(m,n)$ as a Morita algebra by fixing the idempotent 
$
\epsilon = \left(\begin{matrix}
I_m  &  0_{m\times n}\\
0_{n\times m}   &  0_n\\
\end{matrix}
\right)
$, 
\ie, $M(m,n)$ is the Morita algebra corresponding to the Morita context associated to the superspace $U = \FF^m \oplus \FF^n$.

\begin{defi}
	A Morita context $(R, S, M, N, \vphi, \psi )$ is said to be $G$-\emph{graded} if the algebras $R$ and $S$ are graded, the bimodules $M$ and $N$ are graded, and the maps $\vphi$ and $\psi$ are homogeneous of degree $e$. A Morita algebra $(C,\epsilon)$ is said to be $G$-\emph{graded} if $C$ is $G$-graded and $\epsilon$ is a homogeneous element (necessarily of degree $e$).
\end{defi}

Clearly, a Morita context is graded if, and only if, the corresponding Morita algebra is graded.

\begin{remark}\label{remarkk}
For every graded Morita algebra $(C,\epsilon)$, we can define a $\ZZ$-grading by taking $C^{-1} = \epsilon C (1-\epsilon )$, $C^0 = \epsilon C \epsilon \oplus (1-\epsilon)C(1-\epsilon)$ and $C^1=(1-\epsilon)C\epsilon$. In the case of $M(m,n)$, this is precisely the canonical $\ZZ$-grading.
\end{remark}

\begin{prop}\label{prop:3-equiv-even-morita-action}
		Let $\Gamma$ be a grading on the superalgebra $M(m,n)$. The following are equivalent:
		\begin{enumerate}[(i)]
			\item $\Gamma$ is compatible with the canonical $\ZZ$-grading;
			\item $\Gamma$ is even;
			\item $M(m,n)$ equipped with $\Gamma$ is a graded Morita algebra.
			\vspace{1mm}

		\par\vbox{\parbox[t]{\linewidth}{Further, if we assume $\Char\FF=0$, the above statements are also equivalent to:}}
		\vspace{1mm}

	 		\item $\Gamma$ corresponds to a $\widehat G$-action by even automorphisms.
		\end{enumerate}
\end{prop}

\begin{proof}
		 ~\\ \vspace{-2.5mm}

		\textit{(i) $\Rightarrow$ (ii):}
		See the beginning of this subsection.
		\vspace{2mm}

		\textit{(ii) $\Rightarrow$ (iii):}
		Regard $\Gamma$ as a $G^\#$-grading. By Theorem \ref{thm:End-over-D}, there is a graded division algebra $\D$ and a graded right $\D$-module $\mc U$ such that $\End_{\mc D} (\mc U) \simeq M(m,n)$. Take an isomorphism of graded algebras $\phi: \End_{\mc D} (\mc U) \rightarrow M(m,n)$. Since $\Gamma$ is even, $\mc U\even$ and $\mc U\odd$ are graded $\mc D$-submodules. Take $\epsilon' \in \End_{\mc D} (\mc U)$ to be the projection onto $\mc U\even$ associated to the decomposition $\mc U= \mc U\even \oplus \mc U\odd$. Clearly, $\epsilon'$ is a central idempotent of $\End_{\mc D} (\mc U)\even$, hence $\phi(\epsilon')$ is a central idempotent of $M(m,n)\even$, so either $\phi(\epsilon')=\epsilon$ or $\phi(\epsilon')=1-\epsilon$. Either way, $\phi^{-1}(\epsilon)$ is homogeneous, hence so is $\epsilon$.
		\vspace{2mm}

		\textit{(iii) $\Rightarrow$ (i):}
		Follows from Remark \ref{remarkk}.
		\vspace{2mm}

		\textit{(i) $\Leftrightarrow$ (iv):}
		This follows from the fact that the group of even automorphisms is precisely the group of automorphisms of the $\ZZ$-superalgebra structure on $M(m,n)$ (see Remark \ref{rmk:Aut-ZZ-superalgebra}).
\end{proof}

\begin{remark}
	It follows from Proposition \ref{prop:3-equiv-even-morita-action} that, if $\Char \FF = 0$, odd gradings exist only if $m=n$. In Subsection \ref{subsec:odd-gradings} we will give a characteristic-independent proof of this fact.
\end{remark}

We now know that the gradings on the $\ZZ$-superalgebra $M(m,n)$ are precisely the even gradings, but since the automorphism group is different
from the $\ZZ_2$-superalgebra case, the classification of gradings up to isomorphism is also different. 
The proof of the next result is similar to the proof of Theorem \ref{thm:even-assc-iso}.

\begin{thm}
	Let $\Gamma(T,\beta,\gamma_0,\gamma_1)$ and $\Gamma'(T',\beta',\gamma_0',\gamma_1')$ be $G$-gradings on the $\ZZ$-superalgebra $M(m,n)$. Then $\Gamma$ and $\Gamma'$ are isomorphic if, and only if, $T=T'$, $\beta=\beta'$, and there is $g\in G$ such that $g\Xi (\gamma_i) = \Xi (\gamma_i')$ for $i=0,1$. \qed
\end{thm}

As we mentioned in Subsection \ref{subsec:graded-bimodules}, we can always shift the grading on a graded (bi)module and still have a graded (bi)module.
In a graded Morita context, as in the case of a graded superalgebra (see Lemma \ref{lemma:opposite-directions}), we have more structure to preserve: if we shift one of the bimodules by an element $g\in G$ and at least one of the bilinear maps is nonzero, then we are forced to shift the other bimodule by $g^{-1}$. As in the  superalgebra case, we will refer to this situation as \emph{shift in opposite directions}.

\begin{thm}\label{thm:graded-morita}
	Let $\mathcal{C}=(R, S, M, N, \vphi, \psi )$ be the Morita context associated with a superspace $U$ and fix gradings on $R$ and $S$ making them graded algebras. The bimodules $M$ and $N$ admit $G$-gradings so that $\mathcal{C}$ becomes a graded Morita context if, and only if, there exists a graded division algebra $\D$ and graded right $\D$-modules $\mc V$ and $\mc W$ such that $R\iso \End_{\D}(\mc V)$ and $S\iso \End_{\D} (\mc W)$ as graded algebras. 
	Moreover, all such gradings on $M$ and $N$ have the form $M\iso \Hom_{\D}(\mc W, \mc V)^{[g]}$ and $N\iso \Hom_{\D}(\mc V, \mc W)^{[g^{-1}]}$ 
	as graded bimodules, where $g\in G$ is arbitrary.
\end{thm}

\begin{proof}
	Suppose $M$ and $N$ admit $G$-gradings so that the Morita algebra $(C, \epsilon)$ associated to $\mc C$ becomes $G$-graded. By Theorem \ref{thm:End-over-D} there exists a graded division algebra $\D$ and a graded $\D$-module $\mc U$ such that $C \iso \End_{\D} (\mc U)$. Denote the image of $\epsilon$ under this isomorphism by $\epsilon'$ and let $\mc V = \epsilon'(\mc U)$ and $\mc W = (1 -\epsilon')(\mc U)$. Since $\epsilon$ is homogeneous, so is $\epsilon'$, hence $\mc V$ and $\mc W$ are graded $\D$-modules. It follows that $R \iso \epsilon M(m,n) \epsilon \iso \epsilon' \End_{\D} (\mc U) \epsilon' \iso \End_{\D} (\mc V)$ and, analogously, $S \iso \End_{\D} (\mc W)$.

	For the converse, write $C$	in matrix form by fixing a basis in $U$ and identify $\End_{\D} (\mc V)$ and $\End_{\D} (\mc W)$ with matrix algebras as in Definition \ref{def:explicit-grd-assoc}. Suppose there exist isomorphisms of graded algebras $\theta_1 \colon R\rightarrow \End_{\D} (\mc V) $ and $\theta_2 \colon S\rightarrow \End_{\D} (\mc W)$. Then there are $x\in \GL (m)$ and $y\in \GL (n)$ such that $\theta_1$ is the conjugation by $x$ and $\theta_2$ is the conjugation by $y$. It follows that the conjugation by
	$\begin{pmatrix}
		x & 0\\
		0 & y
	\end{pmatrix}$
	\noindent is an isomorphism of algebras between $C\iso M(m,n)$ and
	\[\End_{\D} (\mc V \oplus \mc W) =
	\begin{pmatrix}
		\End_\D(\mc V) & \Hom_\D(\mc W, \mc V)\\
		\Hom_\D(\mc V, \mc W) & \End_\D(\mc W)
	\end{pmatrix},\]
	\noindent hence we transport the gradings on $\Hom_\D(\mc W, \mc V)$ and $\Hom_\D(\mc V, \mc W)$ to	$M$ and $N$, respectively.

	It remains to prove that the gradings on $M$ and $N$ are determined up to shift in opposite directions. Since in our case the Morita algebra $C$ is simple, $M$ and $N$ are simple bimodules.
	By Lemma \ref{lemma:simplebimodule}, the gradings on $M$ and $N$ are determined up to shifts, and the shifts have to be in opposite directions in order for $\vphi$ and $\psi$ to be degree-preserving.
\end{proof}

\subsection{Odd gradings}\label{subsec:odd-gradings}
Let $\Gamma$ be an odd grading on $M(m,n)$. We saw in Subsection \ref{ssec:grds-on-superalgebras} that, as a $G^\#$-graded algebra, $M(m,n)$ is isomorphic to $E\iso \End(\tilde U)\tensor \D$ where the first factor has the trivial $\ZZ_2$-grading and $\D=\D\even\oplus \D\odd$, with $\D\odd\neq 0$, is a $G^\#$-graded division algebra that is simple as an algebra. Let $T\subseteq G^\#$ be the support of $\D$ and $\beta: T\times T \rightarrow \FF^\times$ be the associated bicharacter. 
We write $T^+ = \{t\in T \mid p(t)=\barr 0\} = T \cap G$ and $T^- = \{t\in T \mid p(t)=\barr 1\}$, 
and denote the restriction of $\beta$ to $T^+\times T^+$ by $\beta^+$.

Note that there are no odd gradings if $\Char \FF =2$. Indeed, in this case, there is no nondegenerate bicharacter on $T$ because the characteristic of the field divides $|T|=2|T^+|$. From now on, we suppose $\Char \FF \neq 2$.

For a subgroup $A\subseteq T$, we denote by $A'$ its orthogonal complement in $T$ with respect to $\beta$, i.e., $A' = \{t\in T\mid \beta(t, A) =1\}$. 
This is the inverse image of the subgroup $A^\perp\subseteq \widehat T$ under the isomorphism $T\rightarrow \widehat T$ given by $t\mapsto \beta(t,\cdot)$. 
In particular, $|A'| = [T:A]$.

From these considerations, we have $(T^+)' = \langle t_0 \rangle$ where $t_0$ is an element of order 2. It follows that $\beta(t_0, t) = 1$ if $t\in T^+$ and $\beta(t_0, t) = -1$ if $t\in T^-$. For this reason, we call $t_0$ the \emph{parity element} of the odd grading $\Gamma$. Note that $\rad \beta^+ = T^+\cap (T^+)' = \langle t_0 \rangle$.

Fix an element $0\neq d_0\in \D$  of degree $t_0$. By the definition of $\beta$, $d_0$ commutes with all elements of $\D\even$ and anticommutes with all elements of $\D\odd$. Since $d_0^2\in \D_e = \FF$, we may rescale $d_0$ so that $d_0^2=1$. Then $\epsilon := \frac{1}{2}(1+d_0)$ is a central idempotent of $\D\even$. Take a homogeneous element $0\neq d_1\in\D\odd$. Then $d_1\epsilon d_1\inv = \frac{1}{2}(1-d_0)=1-\epsilon$, which is another central idempotent of $\D\even$ and must have the same rank as $\epsilon$. 
Hence, $\D\even\iso \epsilon\D\even\oplus (1-\epsilon)\D\even$ (direct sum of ideals) and, consequently, $E\even \iso \End(\tilde U)\tensor \D\even = \End(\tilde U)\tensor \epsilon\D\even \oplus \End(\tilde U)\tensor (1-\epsilon)\D\even$, where the two summands have the same dimension. Therefore, odd gradings exist only if $m=n$. 
Also note that we have 
\begin{equation}\label{eq:D1eps}
\D\odd \epsilon = (1-\epsilon) \D\odd.
\end{equation}

We are now going construct an even grading by coarsening a given odd grading. The reverse of this construction will be used in Subsection \ref{ssec:second-odd}.

Let $H$ be a group and suppose we have an even grading $\Gamma'$ on $M(n,n)$ that is the coarsening of $\Gamma$ induced by a group homomorphism $\alpha: G\rightarrow H$. Since $\Gamma'$ is even, then the idempotent $\id_{\tilde U}\tensor\epsilon$ must be homogeneous with respect to $\Gamma'$. This means that $\alpha(t_0)=e$, so $\alpha$ factors through $\barr G := G/\langle t_0 \rangle$. This motivates the following definition:

\begin{defi}
	Let $\Gamma$ be an odd $G$-grading on $M(n,n)$ with parity element $t_0$. The \emph{finest even coarsening of $\Gamma$} is the $\barr G$-grading ${}^\theta \Gamma$, where $\barr G := G/\langle t_0 \rangle$ and $\theta: G \to \barr G$ is the natural homomorphism.
\end{defi}

\begin{thm}
	Let $\Gamma = \Gamma(T, \beta, \gamma)$ be an odd grading on $M(n,n)$ with parity element $t_0$. Then its finest even coarsening is isomorphic to $\barr \Gamma = \Gamma(\barr T, \barr \beta, \barr \gamma, \barr u\barr \gamma)$, where $\barr T= \frac{T^+}{\langle t_0 \rangle}$, $\barr\beta$ is the nondegenerate bicharacter on $\barr T$ induced by $\beta^+$, $\barr\gamma$ is the tuple whose entries are the images of the entries of $\gamma$ under $\theta$, and $u \in G$ is 
	any element such that $(u, \barr 1) \in T^-$.
\end{thm}

\begin{proof}
	Let us focus our attention on the $G$-graded division algebra $\D$. We now consider it as a $\barr G$-graded algebra, which has a decomposition $\D=\D\epsilon \oplus \D(1-\epsilon)$ as a graded left module over itself.

	\setcounter{claim}{0}
	\begin{claim}
		The $\D$-module $\D\epsilon$ is simple as a graded module.
	\end{claim}

	To see this, consider a nontrivial graded submodule $V\subseteq \D\epsilon$ and take a homogeneous element $0\neq v\in V$. Then we can write $v=d\epsilon$ where $d$ is a $\barr G$-homogeneous element of $\D$, so $d = d' + \lambda d' d_0$ where $d'$ is a $G$-homogeneous element and $\lambda\in \FF$. Hence, $v = d'\epsilon + \lambda d'd_0\epsilon = (1+\lambda)d'\epsilon$. Clearly, $(1+\lambda)d'\neq 0$, so it has an inverse in $\D$. We conclude that $\epsilon\in V$, hence $V=\D\epsilon$.\qedclaim 

	Let $\barr \D := \epsilon \D \epsilon \iso \End_{\D}(\D\epsilon)$, where we are using the convention of writing endomorphisms of a left module on the right. By Claim 1 and the graded analog of Schur's Lemma (see \eg \cite[Lemma 2.4]{livromicha}), $\barr \D$ is a $\barr G$-graded division algebra.

	\begin{claim}
		The support of $\barr \D$ is $\barr T= \frac{T^+}{\langle t_0 \rangle}$ and the bicharacter $\barr \beta: \barr T\times \barr T\rightarrow \FF^\times$ is induced by $\beta^+: T^+\times T^+ \rightarrow \FF^\times$.
	\end{claim}

	We have $\barr \D = \epsilon \D\even \epsilon + \epsilon \D\odd \epsilon$ and $\epsilon \D\odd \epsilon = 0$ by Equation \eqref{eq:D1eps}, so $\supp \barr \D \subseteq \barr T$. On the other hand, for every $0\neq d\in \D\even$ with $G$-degree $t\in T^+$, we have that $\epsilon d\epsilon = d\epsilon = \frac{1}{2}(d+dd_0)\neq 0$, since the component of degree $t$ is different from zero. Hence $\supp \barr \D = \barr T$. Since $\epsilon$ is central in $\D\even$, we obtain $\barr\beta (\barr t,\barr s) = \beta (s, t) = \beta^+ (s, t)$ for all $t, s\in T^+$.\qedclaim 

	We now consider $\D\epsilon$ as a graded right $\barr \D$-module. Then we have the decomposition $\D\epsilon = \epsilon \D\epsilon \oplus (1-\epsilon) \D\epsilon$. The set $\{\epsilon\}$ is clearly a basis of $\epsilon \D\epsilon$. To find a basis for $(1-\epsilon)\D\epsilon$, fix any $G$-homogeneous $0\neq d_1\in \D\odd$  with $\deg d_1 = t_1\in T^-$. Then we have $(1-\epsilon)\D\epsilon = (1-\epsilon)\D\even \epsilon + (1-\epsilon)\D\odd \epsilon = (1-\epsilon)\D\odd \epsilon = \D\odd \epsilon$ by Equation \eqref{eq:D1eps}. Since $d_1$ is invertible, $\{d_1\epsilon\}$ is a basis for $(1-\epsilon) \D\epsilon$. We conclude that $\{\epsilon, d_1\epsilon\}$ is a basis for $\D\epsilon$.

	Using the graded analog of the Density Theorem (see e.g. \cite[Theorem 2.5]{livromicha}), we have $\D\iso \End_{\barr \D}(\D\epsilon)\iso \End(\FF\epsilon\oplus \FF d_1\epsilon)\tensor \barr\D$. Hence,
	\[
	\begin{split}
	\End_\D(\mc U)&\iso\End (\tilde U) \tensor \D \iso \End (\tilde U) \tensor \End(\FF\epsilon \oplus \FF d_1\epsilon) \tensor \barr\D \\
	&\iso \End(\tilde U\tensor \epsilon \oplus \tilde U\tensor d_1\epsilon) \tensor \barr\D
	\end{split}
	\]
	as $\barr G$-graded algebras. The result follows.
\end{proof}

In the next section, we will show how to recover $\Gamma$ from $\barr\Gamma$ and some extra data. The following definition and result will be used there.

\begin{defi}
	For every abelian group $A$ we put $A^{[2]} = \{a^2 \mid a\in A\}$ and $A_{[2]} = \{a\in A \mid a^2 = e \}$.
\end{defi}

Note that $T^{[2]}\subseteq T^+$, but $T^{[2]}$ can be larger than $(T^+)^{[2]}$ since it also includes the squares of elements of $T^-$. Also, the subgroup $\barr S = \{\barr t \in \barr T \mid t \in T^{[2]}\}$ of $\barr T$ can be larger than $\barr T^{[2]}$, but we will show that, surprisingly, it does not depend on $T^-$.

\begin{lemma}\label{lemma:square-subgroup}
	Let $\theta: T^+\rightarrow \barr T=\frac{T^+}{\langle t_0 \rangle}$ be the natural homomorphism. 
	Consider the subgroups $\barr S = \theta(T^{[2]})$ and $\barr R=\theta(T^+_{[2]})$ of $\barr T$. 
	Then $\barr S$ is the orthogonal complement of $\barr R$ with respect to the nondegenerate bicharacter $\barr\beta$.
\end{lemma}

\begin{proof}
	We claim that $\barr S' = \barr R$. Indeed,
	\[
		\begin{split}
			\barr S' & = \{ \theta(t) \mid t\in T^+ \AND \barr\beta(\theta (t), \theta (s^2)) =1 \text{ for all }s\in T\}\\ & = \{\theta(t) \mid t\in T^+ \AND \beta (t, s^2) =1 \text{ for all }s\in T\}\\ & = \{ \theta(t) \mid t\in T^+ \AND \beta (t^2, s) =1 \text{ for all }s\in T\}\\ & = \{ \theta(t) \mid t\in T^+ \AND t^2=e \}\\ & = \barr R\,.
		\end{split}
	\]
	It follows that $\barr S = \barr R'$, as desired.
\end{proof}

\subsection{A description of odd gradings in terms of $G$}\label{ssec:second-odd}

Our second description of an odd grading consists of its finest even coarsening and the data necessary to recover the odd grading from this coarsening. All parameters will be obtained in terms of $G$ rather than its extension $G^\#=G\times \ZZ_2$.

Let $t_0\in G$ be an arbitrarily fixed element of order 2 and set $\barr G = \frac{G}{\langle t_0 \rangle}$. Let $\barr T \subseteq \barr G$ be a finite subgroup and let $\barr \beta: \barr T \times \barr T \rightarrow \FF^\times$ be a nondegenerate alternating bicharacter. We define $T^+\subseteq G$ to be the inverse image of $\barr T$ under the natural homomorphism $\theta: G\rightarrow \barr G$. Note that $\barr \beta$ gives rise to a bicharacter $\beta^+$ on $T^+$ whose radical is generated by the element $t_0$. We wish to define $T^-\subseteq G\times \{\barr 1\}$ so that $T=T^+\cup T^-$ is a subgroup of $G^\#$ and $\beta^+$ extends to a nondegenerate alternating bicharacter on $T$.

From Lemma \ref{lemma:square-subgroup}, we have a necessary condition for the existence of such $T^-$, namely, for $\barr R=\frac{T^+_{[2]}}{\langle t_0 \rangle}$, we need $\barr R' \subseteq \barr G^{[2]}$ (indeed, $\barr S$ is a subgroup of $\overline {G^{[2]}} = \barr G^{[2]}$). We will now prove that this condition is also sufficient.

\begin{prop}\label{prop:square-subgroup-converse}
	If $\left( \frac{T^+_{[2]} }{\langle t_0 \rangle}\right)'\subseteq \barr G^{[2]}$, then there exists an element $t_1\in G\times \{\barr 1\} \subseteq G^\#$ such that $T= T^+ \cup t_1\, T^+$ is a subgroup of $G^\#$ and  $\beta^+$ extends to a nondegenerate alternating bicharacter $\beta:T\times T\rightarrow \FF^\times$.
\end{prop}

\begin{proof}
	Let $\chi\in \widehat {T^+}$ be such that $\chi(t_0) = -1$. Since $\chi^2(t_0)=1$, we can consider $\chi^2$ as a character of the group $\barr T = \frac{T^+}{\langle t_0 \rangle}$, hence there is $a\in T^+$ such that $\chi^2(\barr t) = \barr\beta(\barr a, \barr t)$ for all $\barr t\in \barr T$. Note that $\chi (a) = \pm 1$ and hence, changing $a$ to $a t_0$ if necessary, we may assume $\chi (a) = 1$.
	
	\bigskip 

	\textit{(i) Existence of $t_1$}:
	
	\medskip 

	As before, let $\barr R = \frac{T^+_{[2]}}{\langle t_0 \rangle}$. Then $\barr a \in \barr R'$. Indeed, if $b\in T^+_{[2]}$, then $\barr\beta(\barr a,\barr b) = \chi^2 (\barr b) = \chi (b^2) = \chi (e) =1$. By our assumption, we conclude that $\barr a\in \barr G^{[2]}$. We are going to prove that, actually, $a\in G^{[2]}$. Pick $u\in G$ such that $\barr u^2 = \barr a$. Then, either $a=u^2$ or $a=u^2t_0$. If $t_0 = c^2$ for some $c\in G$, then replacing $u$ by $uc$ if necessary, we can make $u^2 = a$. Otherwise, $t_0$ has no square root in $T^+$, which implies that $\barr R=\barr T_{[2]}$. Hence $\barr R' = (\barr T_{[2]})' = \barr T^{[2]} = \theta ((T^+)^{[2]})$. Thus, in this case, we can assume $u\in T^+$. Then $\chi(u^2) = \chi^2(u) = \barr \beta (\barr a, \barr u) = \barr \beta (\barr u^2, \barr u) =1$, hence $u^2 = a$. Finally, we set $t_1=(u,\barr 1) \in G^\#$.

	\bigskip 

	\textit{(ii) Existence of $\beta$}:
	
	\medskip 

	We wish to extend $\beta^+$ to $T=T^+ \cup t_1\, T^+$ by setting $\beta(t_1, t) = \chi (t)$ for all $t\in T^+$. It is clear that there is at most one alternating bicharacter on $T$ with this property that extends $\beta^+$. To show that it exists and is nondegenerate, we will first introduce an auxiliary group $\widetilde T$ and a bicharacter $\tilde\beta$.

	Let $\widetilde T$ be the direct product of $T^+$ and the infinite cyclic group generated by a new symbol $\tau$. We define $\tilde\beta:\widetilde T\times \widetilde T \rightarrow \FF^\times$ by $ \tilde\beta(s\tau^i,t\tau^j) = \beta^+(s,t)\, \chi (s)^{-j}\, \chi (t)^i$, where $s,t\in T^+$. It is clear that $\tilde\beta$ is an alternating bicharacter.

	\begin{claim*}
	$\langle a\tau^{-2} \rangle = \rad \tilde \beta\,$.
	\end{claim*}

	Let $t\in T^+$ and $\ell\in \ZZ$. Then
	\[
	\tilde \beta (a\tau^{-2},t\tau^\ell) =
	\beta^+(a, t)\,\, \chi(t)^{-2} \, \chi(a)^{-\ell} = 
	\barr\beta(\barr a, \barr t)\,\, \chi(t)^{-2} = \chi(t)^2 \, \chi(t)^{-2} = 1,
    \]
    hence, $\langle a\tau^{-2} \rangle \subseteq \rad \tilde \beta$.

    Conversely, if $s\tau^k \in \rad \tilde\beta$, then, $1 = \tilde \beta (s\tau^k, t_0) = \beta^+(s,t_0)\, \chi(t_0)^k = (-1)^k$, hence $k$ is even. From the previous paragraph, we know that $a\tau^{-2} \in \rad \tilde\beta$, hence $a^\frac{k}{2} \tau^{-k} \in \rad \tilde\beta$ and $s a^\frac{k}{2} = (s \tau^k) (a^\frac{k}{2} \tau^{-k}) \in \rad \tilde\beta$. Since $s a^\frac{k}{2} \in T^+$, we get $s a^\frac{k}{2} \in \rad \beta^+ = \{ e, t_0 \}$. But, if $sa^\frac{k}{2} = t_0$, we have $1 = \tilde\beta (sa^\frac{k}{2}, \tau) = \tilde\beta (t_0, \tau) = \chi(t_0)\inv = -1$, a contradiction. It follows that $sa^\frac{k}{2} = e$ and, hence, $s\tau^k = a^{-\frac{k}{2}}\tau^k = (a\tau^{-2})^{\frac{k}{2}}$, concluding the proof of the claim.
    \qedclaim 

    We have a homomorphism $\vphi:\widetilde T\rightarrow T$ that is the identity on $T^+$ and sends $\tau$ to $t_1$. Clearly, $\ker \vphi = \langle a\tau^{-2} \rangle$. By the above claim, $\tilde\beta$ induces a nondegenerate alternating bicharacter on $\frac{\widetilde T}{\langle a\tau^{-2} \rangle}$, which can be transferred via $\vphi$ to a nondegenerate alternating bicharacter on $T$ that extends $\beta^+$.
\end{proof}

Now fix $\chi\in \widehat {T^+}$ with $\chi(t_0)=-1$ and let $a$ be the unique element of $T^+$ such that $\chi(a)=1$ and $\chi^2(\barr t) = \barr\beta (\barr a, \barr t)$ for all $t\in T^+$. Suppose that the condition of Proposition \ref{prop:square-subgroup-converse} is satisfied. 
Then part (i) of the proof shows that there exists $u\in G$ such that $u^2=a$. Moreover, part (ii) shows that there exists an extension of $\beta^+$ to 
a nondegenerate alternating bicharacter $\beta$ on $T=T^+\cup t_1T^+$, where $t_1=(u,\bar 1)$, such that $\beta(t_1,t)=\chi(t)$ for all $t\in T^+$.
Clearly, such an extension is unique. We will denote it by $\beta_u$ and its domain by $T_u$.

\begin{prop}\label{prop:roots-of-a}
For every $T\subseteq G^\#$ such that $T\subsetneq G$ and $T\cap G=T^+$ and for every extension of $\beta^+$ to a nondegenerate alternating bicharacter $\beta$
on $T$, there exists $u\in G$ such that $u^2=a$, $T=T_u$ and $\beta=\beta_u$.
Moreover, $\beta_u=\beta_{\tilde{u}}$ if, and only if, $u \equiv \tilde{u} \pmod{\langle t_0 \rangle}$.
\end{prop}

\begin{proof}

We have $T=T^+ \cup T^-$ where $T^-\subseteq G\times \{\barr 1\}$ is a coset of $T^+$.
We can extend $\chi$ to a character of $T$, which we still denote by $\chi$, and, since $\beta$ is nondegenerate, 
there is $t_1\in T$ such that $\beta(t_1, t) = \chi(t)$ for all $t\in T$. We have $t_1\in T^-$ since $\beta(t_1,t_0)=\chi(t_0)=-1$, so $t_1=(u,\bar 1)$, 
for some $u\in G$, and hence $T=T_u$. We claim that $t_1^2=a$. Indeed, $\chi(t_1^2) = \beta(t_1,t_1^2)=1$ and, for every $t\in T^+$,
\[
 	\chi^2(\barr t) = \chi(t)^2 = \beta (t_1, t)^2 = \beta (t_1^2, t) = \barr\beta (\,\overline {(t_1^2)},\, \barr t)\,,
\]
so $t_1^2$ satisfies the definition of the element $a$. This completes the proof of the first assertion.

Now suppose $\beta_u=\beta_{\tilde{u}}$, so in particular $t_1\,T^+=\tilde{t}_1\,T^+$ where $t_1 = (u, \barr 1)$ and $\tilde{t}_1 = (\tilde u, \barr 1)$.
Then there is $r\in T^+$ such that $\tilde{t}_1 = t_1\,r$. Also, for every $t\in T^+$,
\[
\chi(t) = \beta_{\tilde{u}}(\tilde{t}_1,t) = \beta_u (t_1\,r, t)
		= \beta_u(t_1, t)\,\beta_u(r,t) = \chi(t) \beta^+(r, t)
\]
and, hence, $\beta^+(r, t)=1$ for all $t\in T^+$. This means that $r = u\inv \tilde{u} \in \langle t_0 \rangle$.

Conversely, if $\tilde u = u r$ for some $r\in \langle t_0 \rangle$, then $t_1\, T^+ = \tilde t_1\, T^+$. Also, for all $t\in T^+$,
\[
\beta_u(t_1, t) = \chi(t) = \beta_{\tilde{u}}(\tilde{t}_1, t) = \beta_{\tilde{u}}(t_1r, t) = 
\beta_{\tilde{u}}(t_1, t)\, \beta^+(r, t) = \beta_{\tilde{u}}(t_1, t).
\]
It follows that $\beta_u=\beta_{\tilde{u}}$.
\end{proof}

Note that, keeping the character $\chi \in \widehat {T^+}$ with $\chi(t_0) = -1$ fixed, we have a surjective map from the square roots of $a$ to all possible pairs $(T,\beta)$. If we had started with a different character above, we would have obtained a different surjective map. Hence, for parametrization purposes, $\chi$ (and, hence, $a$) will be fixed.

We are now in a position to give a classification of odd gradings in terms of $G$ only. We already have the following parameters: an element $t_0\in G$ of order $2$ and a pair $(\barr T, \barr\beta)$. For each $t_0$ and $\barr T$, we fix a character $\chi\in \widehat {T^+}$ satisfying $\chi(t_0) = -1$. The next parameter is an element $u\in G$ such that $u^2 = a$, where $a$ is the unique element of $T^+$ such that $\chi(a)=1$ and $\chi^2(\barr t) = \barr\beta (\barr a, \barr t)$ for all $t\in T^+$. Finally, let $\gamma = (g_1, \ldots, g_k)$ be a $k$-tuple of elements of $G$. With these data, we construct the grading $\Gamma (t_0, \barr T, \barr \beta, u, \gamma)$ as follows:

\begin{defi}\label{def:odd-grd-on-Mmn-2}
	Let $\D$ be a standard realization of the $G^\#$-graded division algebra with parameters $(T_u,\beta_u)$. Take the graded $\D$-module $\mathcal U = \D^{[g_1]}\oplus \cdots \oplus \D^{[g_k]}$. Then $\End_\D (\mathcal U)$ is a $G^\#$-graded algebra, hence a superalgebra by means of $p:G^\# \rightarrow \ZZ_2$. As a superalgebra, it is isomorphic to $M(n,n)$ where $n=k\sqrt{|\barr T|}$. We define $\Gamma (t_0, \barr T, \barr \beta, u, \gamma)$ as the corresponding $G$-grading on $M(n,n)$.
\end{defi}

Theorem \ref{thm:first-odd-iso} together with Proposition \ref{prop:roots-of-a} give the following result:

\begin{thm}\label{thm:2nd-odd-iso}
	Every odd $G$-grading on the superalgebra $M(n,n)$ is isomorphic to some $\Gamma (t_0, \barr T, \barr \beta, u, \gamma)$ as in Definition \ref{def:odd-grd-on-Mmn-2}.
	Two odd gradings, $\Gamma (t_0, \barr T, \barr \beta, u, \gamma)$ and $\Gamma (t_0', \barr T', \barr \beta', u', \gamma')$, 
	are isomorphic if, and only if, $t_0=t_0'$, $\barr T = \barr T'$, $\barr\beta = \barr\beta'$, $u \equiv u' \pmod{\langle t_0 \rangle}$, 
	and there is $g\in G$ such that $g\, \Xi(\gamma) = \Xi(\gamma')$.\qed
\end{thm}

\subsection{Fine gradings up to equivalence}

We start by investigating the gradings on the superalgebra $M(m,n)$ that are fine among even gradings. 
By Proposition \ref{prop:3-equiv-even-morita-action}, this is the same as fine gradings on $M(m,n)$ as a $\ZZ$-superalgebra, and, 
by the discussion in Subsection \ref{subsec:odd-gradings}, the same as fine gradings if $m\neq n$ or $\Char \FF = 2$.

We will use the following notation. Let $H$ be a finite abelian group whose order is not divisible by $\Char \FF$. Set $T_H = H\times \widehat H$
and define $\beta_H: T_H\times T_H \to \FF^\times$ by 
\[
\beta_H((h_1, \chi_1), (h_2, \chi_2)) = \chi_1(h_2)\, \chi_2 (h_1)\inv.
\] 
Then $\beta_H$ is a nondegenerate alternating bicharacter on $T_H$. 

\begin{defi}\label{def:even-fine-grd-on-Mmn}
	Let $\ell \mid \operatorname{gcd}(m,n)$ be a natural number such that $\Char \FF\nmid\ell$ and put $k_0 := \frac{m}{\ell}$ and $k_1 := \frac{n}{\ell}$.
	Let $\Theta_\ell$ be a set of representatives of the isomorphism classes of abelian groups of order $\ell$. 
	For every $H$ in $\Theta_\ell$, we define $\Gamma(H, k_0, k_1)$ to be the even $T_H\times \ZZ^{k_0 + k_1}$-grading 
	$\Gamma(T_H, \beta_H, (e_1, \ldots, e_{k_0}), (e_{k_0 + 1}, \dots, e_{k_0 + k_1}))$ on $M(m,n)$, 
	where $\{e_1, \ldots, e_{k_0+k_1}\}$ is the standard basis of $\ZZ^{k_0 + k_1}$. 
	If $m$ and $n$ are clear from the context, we will simply write $\Gamma(H)$.
\end{defi}

Let $G_H$ be the subgroup of $T_H\times \ZZ^{k_0 + k_1}$ generated by the support of $\Gamma(H,k_0,k_1)$, i.e., 
$G_H = T_H\times \ZZ^{k_0 + k_1}_0$, where $\ZZ^k_0 := \{ (x_1, \ldots, x_k) \in \ZZ^k \mid x_1 + \cdots + x_k = 0\}$.

\begin{thm}\label{thm:class-fine-even}
	The fine gradings on $M(m,n)$ as a $\ZZ$-superalgebra are precisely the even fine gradings.
	Every such grading is equivalent to a unique $\Gamma(H)$ as in Definition \ref{def:even-fine-grd-on-Mmn}. 
	Moreover, every grading $\Gamma(H)$ is fine, and $G_H$ is its universal group.
\end{thm}

\begin{proof}
	By \cite[Proposition 2.35]{livromicha}, if we consider $\Gamma(H)$ as a grading on the algebra $M_{n+m}(\FF)$, it is a fine grading and 
	$G_H$ is its universal group. It follows that the same is true of $\Gamma(H)$ as a grading on the superalgebra $M(m,n)$.

	Let $\Gamma = \Gamma(T, \beta, \gamma_0, \gamma_1)$ be any even $G$-grading on $M(m,n)$. We can write $T = A\times B$ where the restrictions of $\beta$ to 
	the subgroups $A$ and $B$ are trivial and, hence, there is an isomorphism $\alpha: T_A \to T$ such that $\beta_A=\beta\circ(\alpha\times\alpha)$. We can extend $\alpha$ to a homomorphism $G_A \to G$ (also denoted by $\alpha$) by sending the elements $e_1, \ldots, e_{k_0}$ to the entries of $\gamma_0$ and the elements $e_{k_0+1}, \ldots, e_{k_0+k_1}$ to the entries of $\gamma_1$. It follows that ${}^\alpha \Gamma(A) \iso \Gamma$. Since all $\Gamma(H)$ are fine and pairwise nonequivalent (because their universal groups are pairwise nonisomorphic), we can apply Lemma \ref{lemma:universal-grp}, concluding that every fine grading 
	on $M(m,n)$ as a $\ZZ$-superalgebra is equivalent to a unique $\Gamma(H)$.
\end{proof}

We now consider odd fine gradings on $M(n,n)$, so $\Char\FF\ne 2$. We first define some gradings on the algebra $M_{2n}(\FF)$ and then impose a superalgebra structure.

\begin{defi}\label{def:param-fine-odd}
	Let $\ell\mid n$ be a natural number such that $\Char \FF\nmid\ell$ and put $k:= \frac{n}{\ell}$.
	Let $\Theta_{2\ell}$ be a set of representatives of the isomorphism classes of abelian groups of order $2\ell$.
	For every $H$ in $\Theta_{2\ell}$, we consider the $T_H\times \ZZ^k$-grading $\Gamma = \Gamma(T_H, \beta_H, (e_1, \ldots, e_k))$ on $M_{2n}(\FF)$,
	where $\{e_1, \ldots, e_k\}$ is the standard basis of $\ZZ^k$.
	Then we choose an element $t_0 \in T$ of order $2$ and define a group homomorphism $p: T_H\times\ZZ^k \to \ZZ_2$ by
	\[
	p(t, x_1, \ldots, x_k) =
	\begin{cases*}
		\bar 0 & if $\beta(t_0, t) = 1$,\\
		\bar 1 & if $\beta(t_0, t) = -1$.
	\end{cases*}
	\]
	This defines a superalgebra structure on $M_{2n}(\FF)$. By construction, $\Gamma$ is odd as a grading on this superalgebra $(M_{2n}(\FF),p)$, 
	and this forces the superalgebra to be isomorphic to $M(n,n)$. We denote by $\Gamma(H, t_0, k)$ the grading $\Gamma$ considered as a grading on $M(n,n)$. 
	If $n$ is clear from the context, we will simply write $\Gamma(H, t_0)$.
\end{defi}

Note that the parameter $t_0$ of $\Gamma(H, t_0, k)$ does not affect the grading on the algebra $M_{2n}(\FF)$, but, as we will see in Proposition \ref{prop:equiv-with-same-H}, different choices of $t_0$ can yield nonequivalent gradings on the superalgebra $M(n,n)$.

\begin{prop}\label{prop:all-fine-odd}
	Each grading $\Gamma(H, t_0)$ on $M(n,n)$ is fine, and its universal group is $G_H = T_H\times \ZZ^k_0$. 
	Every odd fine grading on $M(n,n)$ is equivalent to at least one $\Gamma(H, t_0)$. 
\end{prop}

\begin{proof}
	As in the proof of Theorem \ref{thm:class-fine-even}, the first assertion follows from \cite[Proposition 2.35]{livromicha}. 
	
	Let $\Gamma(T,\beta, \gamma)$ be an odd $G$-grading on $M(n,n)$ and let $t_0$ be its parity element. 
	Then we can find subgroups $A$ and $B$ such that $T=A\times B$ and there exists an isomorphism $\alpha: T_A \to T$ such that $\beta_A=\beta\circ(\alpha\times\alpha)$. We define $t_0' := \alpha\inv(t_0)$ and extend $\alpha$ to a homomorphism $G_A \to G$ (also denoted by $\alpha$)
	by sending the elements $e_1, \ldots, e_k$ to the entries of $\gamma$. Then ${}^\alpha \Gamma(A, t_0') \iso \Gamma$.
	
	Selecting a representative from each equivalence class of gradings of the form $\Gamma(H, t_0)$, we can apply Lemma \ref{lemma:universal-grp}, 
	which proves the second assertion.
\end{proof}

It remains to determine which of the gradings $\Gamma(H, t_0)$ are equivalent to each other.

\begin{prop}\label{prop:equiv-with-same-H}
	The gradings $\Gamma = \Gamma(H, t_0)$ and $\Gamma' = \Gamma(H, t_0')$ on $M(n,n)$ are equivalent if, and only if, there is $\alpha \in \Aut(T_H, \beta_H)$ such that $\alpha(t_0) = t_0'$.
\end{prop}

\begin{proof}
	We will denote by $p: G_H \to \ZZ_2$ the parity homomorphism associated to the grading $\Gamma$ and by 
	$p': G_H\to \ZZ_2$ the one associated to $\Gamma'$.

	If $\Gamma$ is equivalent to $\Gamma'$, there is an isomorphism $\vphi: (M_{2n}(\FF),p) \to (M_{2n}(\FF),p')$ of superalgebras that is a 
	self-equivalence of the grading on $M_{2n}(\FF)$. Hence, we have the corresponding group automorphism $\alpha: G_H \to G_H$ in the Weyl group of the grading, 
	and the following diagram commutes: 
	\begin{equation}\label{diag:parity}
		\begin{tikzcd}
			G_H \arrow[to = 2G, "\alpha"] \arrow[to = Z2, "p"'] && |[alias = 2G]| G_H \arrow[to = Z2, "p'"]\\
			& |[alias = Z2]|\ZZ_2 &
		\end{tikzcd}
	\end{equation}
	By the definition of $p$ and $p'$, this is equivalent to $\alpha(t_0) = \alpha(t_0')$.
	The automorphism $\alpha$ must send the torsion subgroup of $G_H$ to itself, so we can consider the restriction $\alpha\restriction_{T_H}$.
	By \cite[Corrolary 2.45]{livromicha}, this restriction is in $\Aut(T_H, \beta_H)$.

	For the converse, we use the same \cite[Corrolary 2.45]{livromicha} to extend $\alpha$ to an automorphism $G_H\to G_H$ in the Weyl group. 
	Hence, there is an automorphism $\vphi$ of the algebra $M_{2n}(\FF)$ that permutes the components of the grading according to $\alpha$. 
	The condition $\alpha(t_0) = \alpha(t_0')$ is equivalent to Diagram \eqref{diag:parity} being commutative, which shows that $\vphi: (M_{2n}(\FF),p) \to (M_{2n}(\FF),p')$ is an isomorphism of superalgebras.
\end{proof}

Combining Propositions \ref{prop:all-fine-odd} and \ref{prop:equiv-with-same-H}, we obtain:

\begin{thm}\label{thm:class-fine-odd}
	Every odd fine grading on $M(n,n)$ is equivalent to some $\Gamma(H,t_0)$ as in Definition \ref{def:param-fine-odd}.
	Every grading $\Gamma(H,t_0)$ is fine, and $G_H$ is its universal group.
	Two gradings, $\Gamma(H, t_0)$ and $\Gamma(H', t_0')$, are equivalent if, and only if, $H=H'$ and $t_0'$ lies in the orbit of $t_0$
	under the natural action of $\Aut(T_H, \beta_H)$.\qed
\end{thm}

For a matrix description of the group $\Aut(T_H, \beta_H)$, we refer the reader to \cite[Remark 2.46]{livromicha}.

\section{Gradings on $A(m,n)$}\label{sec:Amn}

Throughout this section it will be assumed that $\Char \FF = 0$.

\subsection{The Lie superalgebra $A(m,n)$}\label{ssec:def-A}

Let $U = U\even \oplus U\odd$ be a finite dimentional superspace. Recall that the \emph{general linear Lie superalgebra}, denoted by $\gl\, (U)$, is the superspace $\End(U)$ with product given by the \emph{supercommutator}: 
\[ 
[a,b] = ab - (-1)^{\abs{a}\abs{b}}ba.
\]
If $U\even=\FF^m$ and $U\odd=\FF^n$, then $\gl (U)$ is also denoted by $\gl(m|n)$.

The \emph{special linear Lie superalgebra}, denoted by $\Sl (U)$, is the derived algebra of $\gl(U)$. As in the Lie algebra case, we describe it as an algebra of ``traceless'' operators. The analog of trace in the ``super'' setting is the so called \emph{supertrace}:%
\[ \str \left(\begin{matrix}
	a  &  b\\
	c  &  d\\
	\end{matrix}\right) = \tr a - \tr d,
\]
and we have $\Sl (U) = \{ T\in \gl (U) \mid \Str T = 0\}$.
Again, if $U\even=\FF^m$ and $U\odd=\FF^n$ then $\Sl (U)$ is also denoted by $\Sl(m|n)$.

If one of the parameters $m$ or $n$ is zero, we get a Lie algebra, so we assume this is not the case. If $m\neq n$ then $\Sl(m|n)$ is a simple Lie superalgebra. If $m=n$, the identity map $I_{2n}\in \Sl(n|n)$ is a central element and hence $\Sl(n|n)$ is not simple, but the quotient $\mathfrak{psl}(n|n) := \Sl(n|n)/ \FF I_{2n}$ is simple if $n>1$.

For $m$,$n\geq 0$ (not both zero), the simple Lie superalgebra $A(m,n)$ is $\Sl(m+1|n+1)$ if $m\neq n$, and $\mathfrak{psl}(n+1|n+1)$ if $m=n$.

\begin{defi}\label{def:Type-I}
	If $\Gamma$ is a $G$-grading on $M(m,n)$, then, since $G$ is abelian,  it is also a grading on $\gl(m|n)$ and, hence, restricts to its derived superalgebra $\Sl(m|n)$. If $m=n$, then the grading on $\Sl(m|n)$ induces a grading on $\mathfrak{psl}(n|n)$.
	If a grading on $\Sl(m|n)$ or $\mathfrak{psl}(n|n)$ is obtained in this way, we will call it a \emph{Type I} grading and, otherwise, a \emph{Type II} grading.
\end{defi}

\subsection{Automorphisms of $A(m,n)$}\label{ssec:auto-Amn}
As in the Lie algebra case, the group of automorphisms of the Lie superalgebra $A(m,n)$ is bigger than the group of automorphisms of the associative superalgebra $\End(U)$.

We define the \emph{supertranspose} of a matrix in $\End(U)$ by
\begin{equation*} 
	\left( \begin{matrix}
		a&b\\
		c&d\\
		\end{matrix}\right)^{s\top} = \left( \begin{matrix}
		a\transp & -c\transp\\
		b\transp & d\transp\\
		\end{matrix}\right).
\end{equation*}

The supertranspose map $\End(U) \to \End(U)$ is an example of a \emph{super-anti-automorphism}, \ie, it is $\FF$-linear and
\[
	(XY)^{s\top} = (-1)^{|X||Y|} Y^{s\top} X^{s\top}.
\]
Hence, the map $\tau:\Sl(m+1,n+1)\rightarrow \Sl(m+1,n+1)$ given by $\tau(X) = - X^{s\top}$ is an automorphism.

By \cite[Theorem 1]{serganova}, the group of automorphisms of $A(m,n)$ is generated by $\tau$ and the automorphisms of $\End(U)$, which are restricted to traceless operators and, if necessary, taken modulo the center. In other words, if $m\neq n$, $\Aut(A(m,n))$ is generated by $\mc E \cup \{\tau\}$ and, if $m=n$, by $\mc E \cup \{\pi\,,\,\tau\}$. In both cases, $\mc E$ is a normal subgroup of $\Aut(A(m,n))$. Note that $\pi^2 = \id$, $\tau^2=\upsilon$ (the parity automorphism) and $\pi \tau = \upsilon \tau \pi$. Hence $\frac{\Aut(A)}{\mc E}$ is isomorphic to $\ZZ_{2}$ if $m\neq n$ and $\ZZ_2\times\ZZ_2$ if $m=n$.

Note that a $G$-grading on $A(m,n)$ is of Type I if, and only if, it corresponds to a $\widehat{G}$-action on $A(m,n)$ by automorphisms that belong to 
the subgroup $\mc E$ if $m\ne n$ and to $\mc E\rtimes\langle\pi\rangle$ if $m=n$. If $\widehat{G}$ acts by automorphisms that belong to $\mc E$ then the 
Type I grading is said to be \emph{even} and, otherwise, \emph{odd}.

\subsection{Superdual of a graded module}\label{ssec:superdual}

We will need the following concepts. Let $\D$ be an associative superalgebra with a grading by an abelian group $G$, so we may consider $\D$ graded by the group $G^\# = G\times \ZZ_2$. Let $\U$ be a $G^\#$-graded \emph{right} $\D$-module. The parity $|x|$ of a homogeneous element $x \in \D$ or $x\in \U$ is determined by $\deg x \in G^\#$. The \emph{superdual module} of $\U$ is $\U\Star = \Hom_\D (\U,\D)$, with its natural $G^\#$-grading and the $\D$-action defined on the \emph{left}: if $d \in \D$ and $f \in \U \Star$, then $(df)(u) = d\, f(u)$ for all $u\in \mc U$.

We define the \emph{opposite superalgebra} of $\D$, denoted by $\D\sop$, to be the same graded superspace $\D$, but with a new product $a*b = (-1)^{|a||b|} ba$ for every pair of $\ZZ_2$-homogeneous elements $a,b \in \D$. The left $\D$-module $\U\Star$ can be considered as a right $\D\sop$-module by means of the action defined by $f\cdot d := (-1)^{|d||f|} df$, for every $\ZZ_2$-homogeneous $d\in \D$ and $f\in \U\Star$.

\begin{lemma}\label{lemma:Dsop}
	If $\D$ is a graded division superalgebra associated to the pair $(T,\beta)$, then $\D\sop$ is associated to the pair $(T,\beta\inv)$.\qed
\end{lemma}

If $\U$ has a homogeneous $\D$-basis $\mc B = \{e_1, \ldots, e_k\}$, we can consider its \emph{superdual basis} $\mc B\Star = \{e_1\Star, \ldots, e_k\Star\}$ in $\U\Star$, where $e_i\Star : \U \rightarrow \D$ is defined by $e_i\Star (e_j) = (-1)^{|e_i||e_j|} \delta_{ij}$.

\begin{remark}\label{rmk:gamma-inv}
	The superdual basis is a homogeneous basis of $\U\Star$, with $\deg e_i\Star = (\deg e_i)\inv$. So, if $\gamma = (g_1, \ldots, g_k)$ is the $k$-tuple of degrees of $\mc B$, then $\gamma\inv = (g_1\inv, \ldots, g_k\inv)$ is the $k$-tuple of degrees of $\mc B\Star$.
\end{remark}

For graded right $\D$-modules $\U$ and $\V$, we consider $\U\Star$ and $\V\Star$ as right $\D\sop$-modules as defined above. If $L:\U \rightarrow \V$ is a $\ZZ_2$-homogeneous $\D$-linear map, then the \emph{superadjoint} of $L$ is the $\D\sop$-linear map $L\Star: \V\Star \rightarrow \U\Star$ defined by $L\Star (f) = (-1)^{|L||f|} f \circ L$. We extend the definition of superadjoint to any map in $\Hom_\D (\U, \V)$ by linearity.

\begin{remark}
	In the case $\D=\FF$, if we denote by $[L]$ the matrix of $L$ with respect to the homogeneous bases $\mc B$ of $\U$ and $\mc C$ of $\V$, then the supertranspose $[L]\sT$ is the matrix corresponding to $L\Star$ with respect to the superdual bases $\mc C\Star$ and $\mc B\Star$.
\end{remark}

We denote by $\vphi: \End_\D (\U) \rightarrow \End_{\D\sop} (\U\Star)$ the map $L \mapsto L\Star$. It is clearly a degree-preserving super-anti-isomorphism. It follows that, if we consider the Lie superalgebras $\End_\D (U)^{(-)}$ and $\End_{\D\sop} (U\Star)^{(-)}$, the map $-\vphi$ is an isomorphism.

We summarize these considerations in the following result:

\begin{lemma}\label{lemma:iso-inv}
	If $\Gamma = \Gamma(T,\beta,\gamma)$ and $\Gamma' = \Gamma(T,\beta\inv,\gamma\inv)$ are $G$-gradings (considered as $G^\#$-gradings) on the associative superalgebra $M(m,n)$, then, as gradings on the Lie superalgebra $M(m,n)^{(-)}$, $\Gamma$ and $\Gamma'$ are isomorphic via an automorphism of $M(m,n)^{(-)}$ that is the negative of a super-anti-automorphism of 	$M(m,n)$.
\end{lemma}

\begin{proof}
	Let $\D$ be a graded division superalgebra associated to $(T,\beta)$ and let $\U$ be the graded right $\D$-module associated to $\gamma$. The grading $\Gamma$ is obtained by an identification $\psi: M(m, n) \xrightarrow{\sim} \End_\D (\U)$. By Lemma \ref{lemma:Dsop} and Remark \ref{rmk:gamma-inv}, $\Gamma'$ is obtained by an identification $\psi': M(m, n) \xrightarrow{\sim} \End_{\D\sop} (\U\Star)$. Hence we have the diagram:

	\begin{center}
		\begin{tikzcd}
			& \End_\D (\U) \arrow[to=3-2, "-\vphi"]\\
			M(m, n) \arrow[ur, "\psi"] \arrow[dr, "\psi'"]\\
			& \End_{\D\sop} (\U\Star)
		\end{tikzcd}
	\end{center}

	Thus, the composition $(\psi')\inv \, (-\vphi) \, \psi$ is an automorphism of the Lie superalgebra $M(m,n)^{(-)}$ sending $\Gamma$ to $\Gamma'$.
\end{proof}

\subsection{Type I gradings on $A(m,n)$}

In this work, we only classify the gradings on $A(m,n)$ that are induced from the associative algebra $M(m+1, n+1)$.

\begin{defi}\label{def:grd-on-Amn-I}
	If $\Gamma(T, \beta, \gamma_0,\gamma_1)$ is an even grading on $M(m+1,n+1)$ (see Definition \ref{def:even-grd-on-Mmn}), we denote by $\Gamma_A (T, \beta, \gamma_0,\gamma_1)$ the induced grading on $A(m,n)$. Analogously, if $\Gamma(T, \beta, \gamma)$, or alternatively $\Gamma(t_0, \barr T, \barr\beta, u, \gamma)$, is an odd grading on $M (n+1,n+1)$ (see Definitions \ref{def:odd-grd-on-Mmn-1} and \ref{def:odd-grd-on-Mmn-2}), we denote by $\Gamma_A (T, \beta, \gamma)$, respectively $\Gamma_A (t_0, \barr T, \barr\beta, u, \gamma)$, the induced grading on $A(n,n)$. (Recall that odd gradings can occur only if $m=n$.)
\end{defi}

\begin{thm}\label{thm:even-Lie-iso}
	If a $G$-grading of Type I on the Lie superalgebra $A(m,n)$ is even, then it is isomorphic to some $\Gamma_A(T,\beta, \gamma_0, \gamma_1)$ as in Definition \ref{def:grd-on-Amn-I}.
	Two such gradings, $\Gamma=\Gamma_A(T,\beta, \gamma_0, \gamma_1)$ and $\Gamma'=\Gamma_A (T',\beta', \gamma_0', \gamma_1')$, are isomorphic if, and only if, $T=T'$ and there are $\delta\in \{\pm 1\}$ and $g\in G$ such that $\beta^\delta=\beta'$ and
	\begin{enumerate}[(i)]
		\item for $m\neq n$: $g \Xi (\gamma_0^\delta) =\Xi(\gamma_0')$ and $g \Xi (\gamma_1^\delta) =\Xi(\gamma_1')$;

		\item for $m = n$: either $g \Xi(\gamma_0^\delta)=\Xi(\gamma_0')$ and $g \Xi(\gamma_1^\delta)=\Xi(\gamma_1')$ or $g\Xi(\gamma_0^\delta)=\Xi(\gamma_1')$ and $g \Xi(\gamma_1^\delta)=\Xi(\gamma_0')$.
	\end{enumerate}
\end{thm}

\begin{proof}
	Let $M = M(m+1, n+1)$. Since any automorphism of $M$ induces an automorphism of $A(m,n)$, the first assertion follows from Theorem \ref{thm:even-assc-iso} and the definition of Type I grading.

	We know from Subsection \ref{ssec:auto-Amn} that every automorphism of $A(m, n)$ arises from an automorphism of $M$ or the negative of a super-anti-automorphism of $M$. Moreover, this automorphism or super-anti-automorphism is uniquely determined and, hence, any Type I grading on $A(m,n)$ is induced by a unique grading on $M$. It follows that $\Gamma$ and $\Gamma'$ are isomorphic if, and only if, there exists either $(a)$ an automorphism or $(b)$ a super-anti-automorphism of $M$ sending $\Gamma (T,\beta, \gamma_0, \gamma_1)$ to $\Gamma (T',\beta', \gamma_0', \gamma_1')$.

	From Theorem \ref{thm:even-assc-iso}, we know that case $(a)$ holds if, and only if, the above conditions are satisfied with $\delta = 1$.

	From Lemma \ref{lemma:iso-inv}, there is an automorphism of $A(m,n)$ coming from a super-anti-automorphism of $M$ that sends $\Gamma (T,\beta, \gamma_0, \gamma_1)$ to $\Gamma (T,\beta\inv, \gamma_0\inv, \gamma_1\inv)$. It follows that case $(b)$ holds if, and only if, the above conditions are satisfied with $\delta = -1$.
\end{proof}

\begin{thm}\label{thm:first-odd-Lie-iso}
	If a $G$-grading of Type I on the Lie superalgebra $A(n,n)$ is odd, then it is isomorphic to some $\Gamma_A(T,\beta,\gamma)$ as in Definition \ref{def:grd-on-Amn-I}.
	Two such gradings, $\Gamma_A (T,\beta, \gamma)$ and $\Gamma_A (T',\beta', \gamma')$, are isomorphic if, and only if, $T=T'$, and there are $\delta \in \{\pm 1\}$ and $g\in G$ such that $\beta^\delta=\beta'$ and $g \Xi(\gamma^\delta)=\Xi(\gamma')$.
\end{thm}

\begin{proof}
	The same as for Theorem \ref{thm:even-Lie-iso}, but referring to Theorem \ref{thm:first-odd-iso} instead of Theorem \ref{thm:even-assc-iso}.
\end{proof}

The parameters $T$, $\beta$ and $\gamma$ in Theorem \ref{thm:first-odd-Lie-iso} are associated to the group $G^\#$, not $G$. Below we use parameters associated to $G$, as we did in Subsection \ref{ssec:second-odd}.

\begin{cor}\label{cor:2nd-odd-Lie-iso}
	If a $G$-grading of Type I on the Lie superalgebra $A(n,n)$ is odd, then it is isomorphic to some $\Gamma_A (t_0, \barr T, \barr\beta, u, \gamma)$.
	Two such gradings, $\Gamma_A (t_0, \barr T, \barr \beta, u, \gamma)$ and $\Gamma_A (t_0', \barr T', \barr \beta', u', \gamma')$, are isomorphic if, and only if, $t_0=t_0'$, $\barr T = \barr T'$, and there are $\delta \in \{\pm 1\}$ and $g\in G$ such that $\barr\beta^\delta = \barr\beta'$, $u^\delta \equiv u' \,\, (\operatorname{mod}\,\, \langle t_0 \rangle)$ and $g\, \Xi(\gamma^\delta) = \Xi(\gamma')$.
\end{cor}

\begin{proof}
	Follows from Theorems \ref{thm:first-odd-Lie-iso} and \ref{thm:2nd-odd-iso}.
\end{proof}

\section{Gradings on $P(n)$}\label{sec:Pn}

Throughout this section it will be assumed that $\Char \FF = 0$.

\subsection{The Lie superalgebra $P(n)$}\label{subseq:Pn}

Let $U = U\even \oplus U\odd$ be a superspace and let $\langle\, , \rangle: U\times U\rightarrow \FF$ be a bilinear form that is homogeneous with respect to the $\ZZ_2$-grading, i.e., has parity as a linear map $U\tensor U \rightarrow \FF$.
We say that $\langle\, , \rangle$ is \emph{supersymmetric} if $\langle x,y\rangle = (-1)^{\abs{x}\abs{y}} \langle y,x \rangle$ for all homogeneous elements $x,y\in U$.

From now on, we suppose that $\langle\, , \rangle$ is supersymmetric, nondegenerate, and odd.
The \emph{periplectic Lie superalgebra} $\mathfrak{p}(U)$ is defined as $\mathfrak{p}(U)\even \oplus \mathfrak{p}(U)\odd$ where
\[
	\mathfrak{p}(U)^{i} = \{L\in \gl(U)^i\mid \langle L(x),y\rangle = - (-1)^{i\abs{x}} \langle x,L(y)\rangle\}
\]
for all $i\in\Zmod2$.
The superalgebra $\mathfrak{p}(U)$ is not simple, but its derived superalgebra $P(U) = [\mathfrak{p}(U),\mathfrak{p}(U)]$ is simple if $\Dim U \geq 6$.

Since $\langle\, , \rangle$ is nondegenerate and odd, it is clear that $U\odd$ is isomorphic to $(U\even)^*$ by $u \mapsto \langle u, \cdot\rangle $.
Writing $U\even = V$, we can identify $U$ with $V\oplus V^*$.
Since $\langle \, ,  \rangle$ is supersymmetric, with this identification we have
\[
	\langle v_1+v^*_1,v_2 + v_2^* \rangle =  v_1^* (v_2) + v_2^*(v_1)
\]
for all $v_1, v_2\in V$ and $v_1^*, v_2^*\in V^*$.
Hence, $P(U)$ is a subsuperspace of
\[
	\End(U) = \End(V \oplus V^*) =
	\begin{pmatrix}
		\End (V) & \Hom (V^*, V)\\
		\Hom (V, V^*) & \End(V^*)
	\end{pmatrix}
\]
given by
\[
	P(U) = \left\{\left(\begin{matrix}
	a  &  b        \\
	c  &  -a^*\\
	\end{matrix}
	\right)\Big| \,\tr a = 0,\, b=b^* \AND c=-c^*\right\}.
\]

In the case $V=\FF^{n+1}$, we write $\mathfrak{p}(n)$ for $\mathfrak{p}(U)$ and define $P(n) = [\mathfrak{p}(n),\mathfrak{p}(n)]$, where $n\geq 2$.
Using the standard basis of $V$, we can identify $P(n)$ with the following subsuperalgebra of $M(n+1,n+1)^{(-)}$:
\begin{equation}\label{eq:Pn-abstract}
	P(n) = \left\{\left(\begin{matrix}
	a  &  b   \\
	c  &  -a\transp\\
	\end{matrix}
	\right)\Big| \,\tr a = 0,\, b=b\transp \AND c=-c\transp\right\}.
\end{equation}

One can readily check that $P(U)$ is a graded subspace of $\End (U)$ equipped with its canonical $\ZZ$-grading, so we have $P(U) = P(U)\inv \oplus P(U)^0 \oplus P(U)^1$.
Also, the map $\iota: \Sl(n+1) \rightarrow P(n)^0$ given by
$
	\iota(a) = \left(\begin{matrix}
	a  &  0   \\
	0  &  -a\transp\\
	\end{matrix}
	\right)
$
is an isomorphism of Lie algebras.
If we identify $\Sl(n+1)$ and $P(n)^0$ via this map, then $P(n)^{-1} \iso \mathrm{S}^2 (U\even) \iso V_{2\pi_1}$ and $P(n)^1 \iso \Exterior^2 (U\odd) \iso V_{\pi_{n-1}}$  as modules over $P(n)^0$, where $\pi_i$ denotes the $i$-th fundamental weight of $\Sl(n+1)$.

\subsection{Automorphisms of $P(n)$}

The automorphisms of $P(n)$ were originally described by V.
Serganova (see \cite[Theorem 1]{serganova}).
We give a more explicit description of the automorphism group that we will use for our purposes.

\begin{lemma}\label{lemma:Pn-generates-Mmn}
	Let $U$ be a finite-dimensional superspace equipped with a supersymmetric nondegenerate odd bilinear form.
	The subset $P(U)$ generates  $\End (U)$ as an associative superalgebra.
\end{lemma}

\begin{proof}
	Denote by $R$ the associative superalgebra generated by $P(U)$.
We claim that $U$ is a simple $R$-module.
Indeed, since $P(U)\even\iso \Sl(n+1)$, we have that $U\even \iso V_{\pi_1}$ and $U\odd \iso V_{\pi_n}$ are simple non-isomorphic modules over the Lie algebra $P(U)\even$.
Also, the action of $P(U)\odd$ moves elements from $U\even$ to $U\odd$ and vice-versa, so $U$ does not have nonzero proper subspaces invariant under $P(U)$.

	By Density Theorem, since we are over an algebraically closed field, we conclude that $R = \End (U)$.
\end{proof}

\begin{prop}\label{prop:Aut-Pn}
	The group of automorphisms of $P(n)$ is $\frac{\GL (n+1)}{\{-1,+1\}}$ where $a\in \GL(n+1)$ acts as the conjugation by
	$\left( \begin{matrix}
	a&0\\
	0&(a\transp)^{-1}\\
	\end{matrix}\right)$.
\end{prop}

\begin{proof}
	Let $P=P(n)$ and $\vphi: P \rightarrow P$ be a Lie superalgebra automorphism.
Since it preserves the canonical $\ZZ_2$-grading, taking its restrictions, we obtain a Lie algebra automorphism $\vphi\subeven : P\even \rightarrow P\even$ and an invertible linear map $\vphi\subodd: P\odd \rightarrow P\odd$.

	\setcounter{claim}{0}
	\begin{claim}
		The components $P^{-1}$ and $P^1$ of the canonical $\ZZ$-grading are invariant under $\vphi$.
	\end{claim}

	We denote by $(P\odd)^{\vphi\subeven}$ the $P\even$-module $P\odd$ twisted by $\vphi\subeven$, \ie, the space $P\odd$ with a new action given by $\ell \cdot x = \vphi\subeven (\ell)x$ for all $\ell \in P\even$ and $x \in P\odd$.
Clearly, the map $\vphi\subodd: P\odd \rightarrow (P\odd)^{\vphi\subeven}$ is a $P\even$-module isomorphism.
In particular, $(P\odd)^{\vphi\subeven} = \vphi\subodd (P^{-1}) \oplus \vphi\subodd (P^{1})$, where $\vphi\subodd (P^{-1})$ and $\vphi\subodd (P^{1})$ are simple and non-isomorphic.
It follows that either $(P^{-1})^{\vphi\subeven} = \vphi\subodd (P^{-1})$ or $(P^{-1})^{\vphi\subeven} = \vphi\subodd (P^{1})$.
By dimension count, we have $(P^{-1})^{\vphi\subeven} = \vphi\subodd (P^{-1})$ and, similarly, $(P^{1})^{\vphi\subeven} = \vphi\subodd (P^{1})$.

	\begin{claim}
		The automorphism $\vphi\subeven$ is inner.
	\end{claim}

	If we identify $\Sl(n+1)$ with $P^0$ via the map $\iota$ defined in Subsection \ref{subseq:Pn}, we have $P^{-1}\iso V_{2\pi_1}$ as an $\Sl(n+1)$-module.
	By Claim 1, we know that $\vphi\subodd \restriction_{P\inv} : P\inv \rightarrow (P\inv)^{\vphi\subeven}$ is an isomorphism of modules, but if $\vphi\subeven$ were an outer automorphism, we would have $(V_{2\pi_1})^{\vphi\subeven} \simeq V_{2\pi_n}$, which would force $n=1$, a contradiction.

	\begin{claim}
		If $\varphi\subeven = \id$, then $\varphi = \upsilon_\lambda$	for some $\lambda\in \FF^\times$.
	\end{claim}

	Recall from Subsection \ref{ssec:G-hat-action} that $\upsilon_{\lambda}$ acts as $\lambda^i \id$ on $P^i$.
	Since $\varphi\subeven = \id$, $\vphi_{\bar1}: P\odd \rightarrow P\odd$ is a $P^0$-module automorphism.
	By Claim 1 and Schur's Lemma, $\vphi_{\bar1}\restriction_{P^{-1}}$ and $\vphi_{\bar1}\restriction_{P^{1}}$ are scalar operators.
	Due to the superalgebra structure, these two scalars must be inverses of each other, concluding the proof of the claim.\qedclaim 

	By Claim 2, we know that there is an invertible $a\in \End(U\even)$ such that $\vphi\subeven$ is the conjugation by $A=\left( \begin{matrix}
	a&0\\
	0&(a\transp)^{-1}\\
\end{matrix}\right)$.
By Claim 3, $\vphi$ must be this conjugation composed with $\upsilon_\lambda$	for some $\lambda\in \FF^\times$.
But $\upsilon_\lambda$ is the conjugation by $\left( \begin{matrix}
	\mu\inv\id &0\\
	0&\mu \id\\
	\end{matrix}
\right)$ where $\mu^2=\lambda$, so we can adjust $a$ and assume that $\vphi$ is the conjugation by $A$.
Since $P$ generates $M(n+1,n+1)$ as an associative superalgebra (Lemma \ref{lemma:Pn-generates-Mmn}), $A$ is determined up to scalar and, clearly, the only possible scalars are $-1$ and $1$.
\end{proof}

\begin{remark}
	The images of $\upsilon_\lambda$, $\lambda \in \FF^\times$, cover the group of outer automorphisms of $P(n)$ (see \cite[Theorem 1]{serganova}).
\end{remark}

\subsection{Restriction of gradings from $M(n+1,n+1)$ to $P(n)$}

We start with a consequence of Proposition \ref{prop:Aut-Pn}.

\begin{cor}\label{cor:automorphisms-Pn}
	Every automorphism of $P(n)$ is the restriction of a unique even automorphism of $M(n+1, n+1)$ and every grading on $P(n)$ is the restriction of a unique even grading on $M(n+1, n+1)$.
\end{cor}

\begin{proof}
	Consider the embedding $\Aut(P(n))\rightarrow \Aut(M(n+1,n+1))$ that follows from Proposition \ref{prop:Aut-Pn}.
	The image consists of even automorphisms, so Proposition \ref{prop:3-equiv-even-morita-action}(iv) implies that every $G$-grading on $P(n)$ extends to an even grading on $M(n+1, n+1)$. The uniqueness follows from Lemma \ref{lemma:Pn-generates-Mmn}.
\end{proof}

Of course, not every even grading on $M(n+1,n+1)$ restricts to $P(n)$. 
We are going to obtain necessary and sufficient conditions for such restriction to be possible.

Let $\D$ be a finite-dimensional graded division algebra.
The concept of \emph{dual of a graded $\D$-module} is a special case of the concept of \emph{superdual} discussed in Subsection \ref{ssec:superdual}, 
which arises when the gradings on $\D$ and its graded modules are even.
Furthermore, in our situation $T$ must be an elementary $2$-group (see Theorem \ref{thm:Pn-elem-2-grp}).
Let us recall the definitions and specialize them to the case at hand.

Let $\mc V$ be a \emph{right} graded $\D$-module. Then $\mc V^{\star}=\Hom_{\D} (\mc V, \D)$ is a \emph{left} $\D$-module with the action given by $(d\cdot f) (v) = d f(v)$ for all $d\in \D$, $f\in\mc V^\star$ and  $v\in\mc V$.
If $\mc B = \{v_1, \ldots, v_k\}$ is a homogeneous basis for $\mc V$, the dual basis $\mc B\Star \subseteq \mc V \Star$ consists of the elements $v_i\Star$, $1\leq i \leq k$, defined by $v_i\Star (v_j) = \delta_{ij}$.
Note that $\operatorname {deg} v_i\Star = (\operatorname {deg} v_i)\inv$.
Given two right $\D$-modules, $\mc V$ and $\mc W$, and a $\D$-linear map  $L:\mc V\rightarrow \mc W$, we have the adjoint $L^\star: \mc W^\star\rightarrow \mc V^\star$ defined by $L^\star(f) = f\circ L$, for every $f\in\mc W^\star$.

We now assume that $\D$ is a standard realization associated to a pair $(T, \beta)$ such that $T$ is an elementary $2$-group.
With this we can identify $\D$ with $\D\op$ via transposition (see Remark \ref{rmk:2-grp-transp}) and, hence, we can regard left $\D$-modules as right $\D$-modules.
In particular, if $\mc V$ is a graded right $\D$-module, then $\mc V\Star$ is a graded right $\D$-module via $(f \cdot d)(v) = d\transp f(v)$ for all $f\in \mc V\Star$, $d\in\D$ and $v\in \mc V$.

Consider the space $\Hom_\D(\mc V, \mc W)$.
Fixing homogeneous $\D$-bases $\mc B = \{v_1, \ldots, v_k\}$ and $\mc C = \{w_1, \ldots, w_\ell\}$ for $\mc V$ and $\mc W$, respectively, we obtain an isomorphism between $\Hom_\D(\mc V, \mc W)$ and $\M_{\ell \times k} (\D)$.
The latter is naturally isomorphic to $\M_{\ell \times k} (\FF) \tensor \D$, so we will identify them.

\begin{lemma}\label{lemma:D-transp}
	Let $L: \mc V\rightarrow \mc W$ be a $\D$-linear map.
We fix homogeneous $\D$-bases $\mc B$ and $\mc C$ on $\mc V$ and $\mc W$, respectively, and their dual bases in $\mc V\Star$ and $\mc W\Star$.
If $A\tensor d \in \M_{\ell \times k} (\FF) \tensor \D$ represents $L$, then $A\transp \tensor d\transp$ represents $L\Star$.\qed
\end{lemma}

We can regard the elements of $\M_{\ell \times k} (\FF) \tensor \D$ as matrices over $\FF$ via Kronecker product (as in Definition \ref{def:explicit-grd-assoc}).
Then we have $A\transp \tensor d\transp = (A\tensor d)\transp$.

\begin{thm}\label{thm:Pn-elem-2-grp}
Let $U$ be a finite-dimensional superspace and let $\Gamma = \Gamma (T, \beta, \gamma_0, \gamma_1)$ be an even $G$-grading on $\End(U)$.
The superspace $U$ admits a supersymmetric nondegenerate odd bilinear form such that $P(U)$ is a $G$-graded subsuperalgebra of $\End(U)^{(-)}$ if, and only if, $T$ is an elementary $2$-group and there is $g_0\in G$ such that $\Xi(\gamma_1) = g_0 \, \Xi(\gamma_0\inv)$.
Moreover, if there are two supersymmetric nondegenerate odd bilinear forms on $U$ such that the corresponding $P_1(U)$ and $P_2(U)$ are $G$-graded subsuperalgebras, then $P_1(U)$ and $P_2(U)$ are ismorphic up to shift in opposite directions.
\end{thm}

\begin{proof}
	Assume that, for some form, $P(U)$ is a $G$-graded subsuperalgebra.
	Let $V=U\even$ and consider the identification of $U\odd$ with $V^*$ presented in Subsection \ref{subseq:Pn}.
This way $\Gamma = \Gamma (T, \beta, \gamma_0, \gamma_1)$ is an even grading on
	\[
		\End(U) = \End(V \oplus V^*) =
		\begin{pmatrix}
			\End (V) & \Hom (V^*, V)\\
			\Hom (V, V^*) & \End(V^*)
		\end{pmatrix}.
	\]
	In particular, $\End(V)$ and $\End(V^*)$ are graded subspaces of $\End(U)\even$, with gradings $\Gamma (T, \beta, \gamma_0)$ and 
	$\Gamma (T, \beta, \gamma_1)$, respectively. If
	\[
		x = \left(\begin{matrix}
		a  &  0        \\
		0  &  -a^*\\
		\end{matrix}
		\right)
	\]
	is a homogeneous element in $P(U)\even$, then both $u(x) := a \in \Sl(V) \subseteq \End(V)$ and $v(x) := -a^* \in \Sl(V^*) \subseteq \End(V^*)$ are homogeneous elements of the same degree.
In other words, the maps $u: P(n)\even \rightarrow \Sl(V)$ and $v: P(n)\even \rightarrow \Sl(V^*)$ are homogeneous of degree $e$.
Consider the algebra isomorphism $\vphi: \End(V)\op \rightarrow \End(V^*)$ associating to each operator its adjoint.
Clearly, $\vphi(a) = - (v \circ u\inv) (a)$ for all $a\in \Sl(V)$.
Since $\End(V) = \Sl(V)\, \oplus\, \FF\id_V$ and $\vphi(\id_V) = \id_{V^*}$, we see that $\vphi$ is homogeneous of degree $e$.
From Lemma \ref{lemma:Dsop} and Remark \ref{rmk:gamma-inv}, we conclude that $\Gamma (T, \beta\inv, \gamma_0\inv) \iso \Gamma (T, \beta, \gamma_0)$, and hence, by Theorem \ref{thm:classification-matrix}, $\beta\inv = \beta$ and there is $g_0\in G$ such that $g_0\,\Xi(\gamma_0\inv) = \Xi(\gamma_1)$.
Since $\beta$ is nondegenerate, $\beta\inv = \beta$ if, and only if, $T$ is an elementary $2$-group.

	Note that the $G$-graded algebra $P(U)\even$ is isomorphic (via the map $u$) to the $G$-graded subalgebra $\Sl(V)$ of $\End(V)^{(-)}$, where the grading on $\End(V)$ is $\Gamma(T,\beta, \gamma_0)$.
Therefore, if we have two forms such that the corresponding $P_1(U)$ and $P_2(U)$ are $G$-graded subsuperalgebras, then their even parts are isomorphic as $G$-graded algebras.
Using Lemmas \ref{lemma:simplebimodule} and \ref{lemma:opposite-directions}, we conclude the ``moreover'' part.

	Now assume, conversely, that $T$ is an elementary $2$-group and $\Xi(\gamma_1) = g_0 \, \Xi(\gamma_0\inv)$.
	We can adjust $\gamma_1$, if necessary, so that $\gamma_1 = g_0\, \gamma_0\inv$ and the isomorphism class of $\Gamma$ does not change.

	Let $\D$ be a standard realization of a graded division algebra associated to $(T,\beta)$ and let $\mc V$ be a graded right $\D$-module 
	with a homogeneous basis $\mc B$ whose degrees are given by $\gamma_0$.
	Define $\mc U = \mc U\even \oplus \mc U\odd$ with $\mc U\even = \mc V$ and $\mc U\odd = (\mc V\Star)^{[g_0]}$.
	The $G$-grading $\Gamma$ on $\End(U)$ is defined by means of an isomorphism:
	\[
	\begin{split}
		\End(U) \iso \End_\D (\mc U) &= \begin{pmatrix}
			                  \End_\D (\mc V) & \Hom_\D ((\mc V\Star)^{[g_0]}, \mc V)\\
												\Hom_\D (\mc V, (\mc V\Star)^{[g_0]}) & \End_\D ((\mc V\Star)^{[g_0]})
											\end{pmatrix}\\
											&= \begin{pmatrix}
												                  \End_\D (\mc V) & \Hom_\D (\mc V\Star, \mc V)^{[g_0\inv]}\\
																					\Hom_\D (\mc V, \mc V\Star)^{[g_0]} & \End_\D (\mc V\Star)
																				\end{pmatrix}.
	\end{split}
	\]
	 Using the homogeneous $\D$-bases $\mc B$ for $\mc V$ and $\mc B\Star$ for $\mc V\Star$ to represent $\D$-linear maps by matrices in $M_k(\D) = M_k(\FF) \tensor \D$ and using the Kronecker product to identify the latter with $M_{n+1}(\FF)$, we obtain an isomorphism $\End(U) \xrightarrow{\sim} M(n+1, n+1)$, and $M(n+1, n+1)$ contains $\mathfrak{p}(n)$ and $P(n) = [\mathfrak{p}(n), \mathfrak{p}(n)]$ as in Equation \eqref{eq:Pn-abstract}.

	The above isomorphism $\End(U)\xrightarrow{\sim} M(n+1,n+1)$ of superagebras is given by an isomorphism of superspaces $U \xrightarrow{\sim} \FF^{n+1} \oplus \FF^{n+1}$.
Hence, there exists a supersymmetric nondegenerate odd bilinear form on $U$ such that $P(U)$ corresponds to $P(n)$ under the above isomorphism.

	Finally, we have to show that $P(U)$ is a $G$-graded subsuperspace of $\End(U)$.
Clearly, it is sufficient to prove the same for $\mathfrak{p}(U)$.
But $\mathfrak{p}(U)$ corresponds to
	\begin{equation*}
		\mathfrak{p}(n) = \left\{\left(\begin{matrix}
		a  &  b   \\
		c  &  -a\transp\\
		\end{matrix}
		\right)\Big| \,a,b,c\in M_{n+1}(\FF),\, b=b\transp \AND c=-c\transp\right\} \subseteq M(n+1,n+1),
	\end{equation*}
	which, in view of Lemma \ref{lemma:D-transp}, corresponds to the subsuperspace
	\[\begin{split}
		\bigg \{\left(\begin{matrix}
		a  &  b   \\
		c  &  -a\Star\\
		\end{matrix}
		\right)\Big| \,a\in \End_\D(\mc U),\, b=b\Star\in \Hom_\D(\mc V\Star, \mc V),\, c=-c\Star\in \Hom_\D(\mc V, \mc V\Star) \bigg\}
	\end{split}
	\]
	of $\End_\D(\mc U)$, which is clearly a $G$-graded subsuperspace.
\end{proof}

\subsection{$G$-gradings up to isomorphism}

\begin{defi}\label{def:grd-Pn}
	Let $T\subseteq G$ be a finite elementary $2$-subgroup, $\beta$ be a nondegenerate alternating bicharacter on $T$, $\gamma$ be a $k$-tuple of elements of $G$, and $g_0\in G$. We will denote by $\Gamma_P (T, \beta, \gamma, g_0)$ the grading on the superalgebra $P(n)$ obtained by restricting the grading $\Gamma(T,\beta,\gamma,g_0\gamma\inv)$ on $M(n+1,n+1)$ as in the proof of Theorem \ref{thm:Pn-elem-2-grp}.	
	Explicitly, write $\gamma = (g_1, \ldots, g_k)$ and take a standard realization of a graded division algebra $\D$ associated to $(T, \beta)$.
Then $M_{n+1}(\FF)\iso M_k(\FF) \tensor \D$ by means of Kronecker product, and
	\[
		M(n+1, n+1) \iso
		\begin{pmatrix}
			M_k (\FF)\tensor \D & M_k (\FF)\tensor \D\\
			M_k (\FF)\tensor \D & M_k (\FF)\tensor \D
		\end{pmatrix}
	\]
	Denote by $E_{ij}$ the $(i,j)$-th matrix unit in $M_k (\FF)$. The grading $\Gamma(T,\beta,\gamma,g_0\gamma\inv)$ is given by:
	\begin{center}
	\begin{tabular}{@{$\bullet$ }ll}
	$\deg (E_{ij}\tensor d) = g_i (\deg d) g_j\inv$      & in the upper left corner;\\
	 $\deg (E_{ij}\tensor d) = g_i (\deg d) g_j \, g_0\inv$ & in the upper right corner;\\
	 $\deg (E_{ij}\tensor d) = g_i\inv (\deg d) g_j\inv g_0$   & in the lower left corner;\\
	 $\deg (E_{ij}\tensor d) = g_i\inv (\deg d) g_j$ & in the lower right corner.
	\end{tabular}
	\end{center}
\end{defi}

Note that the restriction of $\Gamma_P(T, \beta, \gamma, g_0)$ to the even part is the inner grading on $\Sl(n+1)$ with parameters $(T, \beta, \gamma)$.

\begin{thm}\label{thm:Pn-iso}
	Every $G$-grading on the Lie superalgebra $P(n)$ is isomorphic to some $\Gamma_P (T, \beta, \gamma, g_0)$ as in Definition \ref{def:grd-Pn}.
	Two gradings, $\Gamma = \Gamma_P (T,\beta,\gamma,g_0)$ and $\Gamma' = \Gamma_P (T',\beta',\gamma',g_0')$, are isomorphic if and only if $T=T'$, $\beta = \beta'$, and there is $g\in G$ such that $g^2 g_0 = g_0'$ and $g\,\Xi(\gamma)=\Xi(\gamma')$.
\end{thm}

\begin{proof}
	The first assertion follows from Corollary \ref{cor:automorphisms-Pn} and Theorem \ref{thm:Pn-elem-2-grp}.
	For the second assertion, recall that $\Gamma$ and $\Gamma'$ are, respectively, the restrictions of the gradings  
	$\widetilde \Gamma = \Gamma (T, \beta, \gamma, g_0 \gamma\inv)$ and $\widetilde \Gamma' = \Gamma (T', \beta', \gamma', g_0' (\gamma')\inv)$
	on $M(n+1, n+1)$. 

	$(\Rightarrow)$:
	Suppose $\Gamma \iso \Gamma'$.
	Since every automorphism of $P(n)$ extends to an automorphism of $M(n+1, n+1)$ (Corollary \ref{cor:automorphisms-Pn}), we have 
	$\widetilde \Gamma \iso \widetilde \Gamma'$, which implies $T=T'$ and $\beta = \beta'$ by Theorem \ref{thm:even-assc-iso}.

	Let $\D$ be a standard realization associated to $(T, \beta)$ and let $\mc V$ be a right $\D$-module with basis $\mc B = \{v_1, \ldots, v_k\}$, 
	which is graded by assigning $\deg v_i = g_i$.
	The same module, but with $\deg v_i = g_i'$, will be denoted by $\mc W$.
	Then $E = \End_\D (\mc V \oplus (\mc V\Star)^{[g_0]})$ and $E' = \End_\D (\mc W \oplus (\mc W\Star)^{[g_0']})$ are graded superalgebras.
	Using the bases $\mc B$ and $\mc B\Star$ and the Kronecker product, we can identify them with $M(n+1, n+1)$.
	The first identification gives the grading $\widetilde \Gamma$ on $M(n+1, n+1)$ and the second gives $\widetilde \Gamma'$.

	Let $\Phi$ be an automorphism of $M(n+1, n+1)$ that sends $\widetilde\Gamma$ to $\widetilde\Gamma'$.
	By Proposition \ref{prop:Aut-Pn}, $\Phi$ is the conjugation by
	\[
	A =
	\begin{pmatrix}
		a & 0\\
		0 & (a\transp)\inv
	\end{pmatrix}
	\]
	for some $a\in \GL(n+1)$.
	By Lemma \ref{lemma:D-transp}, $\Phi$ corresponds to the isomorphism $E\rightarrow E'$ that is the conjugation by
	\[
	\phi =
	\begin{pmatrix}
		\alpha & 0\\
		0 & (\alpha\Star)\inv
	\end{pmatrix}
	\]
	where $\alpha: \mc V\rightarrow \mc W $ and $(\alpha\Star)\inv: (\mc V\Star)^{[g_0]} \rightarrow (\mc W\Star)^{[g_0']} $ are $\D$-linear maps.
	On the other hand, by Proposition \ref{prop:inner-automorphism}, this isomorphism $E\to E'$ is the conjugation by a homogeneous bijective $\D$-linear map
	\[
	\psi=
	\left( \begin{matrix}
		\psi_{11}&\psi_{12}\\
		\psi_{21}&\psi_{22}\\
	\end{matrix}\right).
	\]
	It follows that there is  $\lambda\in \FF$ such that $\phi = \lambda\psi$, and, hence, $\phi$ is homogeneous.
	Let us denote its degree by $g$.
	Then both $\alpha$ and $(\alpha\Star)\inv$ must be homogeneous of degree $g$.
	Hence, $\alpha: \mc V^{[g]}\to\mc W $ is an isomorphism of graded $\D$-modules, so we conclude that $g \Xi(\gamma) = \Xi(\gamma')$.
	Considered as a map $\mc V\Star \rightarrow \mc W\Star$, $(\alpha\Star)\inv$ would have degree $g\inv$, so taking into account the shifts, it has degree $g\inv g_0\inv g_0'$, which must be equal to $g$, so $g_0' = g^2 g_0$.

	$(\Leftarrow)$:
	We may suppose $\D=\D'$. Since $g\Xi(\gamma) = \Xi(\gamma')$, we have an isomorphism of graded $\D$-modules $\alpha: \mc V^{[g]} \rightarrow \mc W$.
	As a map from $\mc V$ to $\mc W$, $\alpha$ is homogeneous of degree $g$, hence $(\alpha\Star)\inv: \mc V\Star \rightarrow \mc W\Star$ has degree $g\inv$.
	It follows that, as a map from $(\mc V\Star)^{[g_0]}$ to $(\mc W\Star)^{[g_0']}$, $(\mc \alpha\Star)\inv$ has degree $g\inv g_0\inv g_0' = g$.
	The desired automorphism of $P(n)$ that sends $\Gamma$ to $\Gamma'$ is the conjugation by the matrix 
	$\psi=\left( \begin{matrix}
	\alpha&0\\
	0&(\alpha\Star)\inv\\
	\end{matrix}\right).$
\end{proof}

\subsection{Fine gradings up to equivalence}

For every integer $\ell\ge 0$, we set $T_{(\ell)}=\ZZ_2^{2\ell}$ and fix a nondegenerate alternating bicharacter $\beta_{(\ell)}$, say,
\[
\beta_{(\ell)} (x,y)=(-1)^{\sum_{i=1}^{2\ell} x_i y_{2\ell-i+1}}.
\]

\begin{defi}\label{def:fine-grd-Pn}
For every $\ell$ such that $2^\ell$ is a divisor of $n+1$, put $k:=\frac{n+1}{2^\ell}$ and $\tilde{G}_{(\ell)}=T_{(\ell)}\times \ZZ^k$. 
Let $\{e_1,\ldots,e_k\}$ be the standard basis of $\ZZ^k$ and let $\langle e_0\rangle$ be the infinite cyclic group generated by a new symbol $e_0$.
We define $\Gamma_P(\ell,k)$ to be the $\tilde{G}_{(\ell)}\times\langle e_0 \rangle$-grading $\Gamma_P (T_{(\ell)},\beta_{(\ell)},(e_1, \ldots, e_k), e_0)$ on $P(n)$.
If $n$ is clear from the context, we will simply write $\Gamma_P(\ell)$.
\end{defi}

The subgroup of $\tilde{G}_{(\ell)} \times \langle e_0 \rangle$ generated by the support of $\Gamma_P(\ell,k)$ is 
\[
G_{(\ell)} := (T_{(\ell)}\times \ZZ^k_0)\oplus \langle 2e_1 - e_0 \rangle \iso \ZZ_2^{2\ell}\times \ZZ^k.
\]

\begin{prop}\label{prop:P-fine}
	The gradings $\Gamma_P(\ell)$ on $P(n)$ are fine and pairwise nonequivalent.
\end{prop}

\begin{proof}
	We can write $\Gamma_P (\ell) = \Gamma^{-1} \oplus \Gamma^0 \oplus \Gamma^1$ where $\Gamma^i$ is the restriction of $\Gamma_P(\ell)$ to 
	the $i$-th component of the canonical $\ZZ$-grading of $P(n)$.
	We identify $ P(n)^0 = P(n)\even$ with $ \Sl(n+1)$ via de map $\iota$ defined in Subsection \ref{subseq:Pn}. 
	Then the grading $\Gamma^0$ on $P(n)^0$ is the restriction to $\Sl(n+1)$ of a fine grading on $M_{n+1}(\FF)$ with universal group $T_{(\ell)}\times\ZZ^k_0$ (\cite[Proposition 2.35]{livromicha}), so it has no proper refinements among the inner gradings on $\Sl(n+1)$.
	Also, $\Gamma_P(\ell)$ and $\Gamma_P(\ell')$ are nonequivalent if $\ell\ne\ell'$, because their restrictions to $P(n)^0$ are nonequivalent.

	Note that the supports of $\Gamma^{-1}$, $\Gamma^0$ and $\Gamma^1$ are pairwise disjoint since they project to, respectively, $-e_0$, $0$, and $e_0$ 
	in the direct summand $\langle e_0 \rangle $ of $\tilde{G}_{(\ell)}\times\langle e_0\rangle$.
	Suppose that the grading $\Gamma_P (\ell)$ admits a refinement $\Delta = \Delta\inv \oplus \Delta^0 \oplus \Delta^1$.
	Then $\Delta^0$ is an inner grading that is a refinement of $\Gamma^0$, hence they are the same grading (up to relabeling).
	Using Lemma \ref{lemma:simplebimodule}, we conclude that $\Gamma$ and $\Delta$ are the same grading, proving that $\Gamma$ is fine.
\end{proof}

\begin{thm}\label{thm:class-fine-Pn}
	Every fine grading on $P(n)$ is equivalent to a unique $\Gamma_P(\ell)$ as in Definition \ref{def:fine-grd-Pn}.
	Moreover, every grading $\Gamma_P(\ell)$ is fine, and $G_{(\ell)}$ is its universal group.
\end{thm}

\begin{proof}
	Let $\Gamma=\Gamma_P(G,T,\beta,\gamma,g_0)$ be any $G$-grading on $P(n)$.
	Since $T$ is an elementary $2$-group of even rank, we have an isomorphism $\alpha:T_{(\ell)}\to T$, for some $\ell$, such that $\beta_{(\ell)}=\beta\circ(\alpha\times\alpha)$. We can extend $\alpha$ to a homomorphism $G_{(\ell)}\rightarrow G$ (also denoted by $\alpha$) by sending the elements $e_1,\ldots,e_k$ to the entries of $\gamma$, and $e_0$ to $g_0$. 
	By construction, ${}^\alpha\Gamma_P {(\ell)}\iso\Gamma$.
	It remains to apply Proposition \ref{prop:P-fine} and Lemma \ref{lemma:universal-grp}. 
\end{proof}

\section*{Acknowledgments}
The first two authors were Ph.D. students at Memorial University of Newfoundland while working on this paper. Helen Samara Dos Santos would like to thank her co-supervisor, Yuri Bahturin, for help and guidance during her Ph.D. program. All authors are grateful to Yuri Bahturin and Alberto Elduque for useful discussions.


 \newcommand{\noop}[1]{} \def\cprime{$'$}
\providecommand{\bysame}{\leavevmode\hbox to3em{\hrulefill}\thinspace}
\providecommand{\MR}{\relax\ifhmode\unskip\space\fi MR }
\providecommand{\MRhref}[2]{%
  \href{http://www.ams.org/mathscinet-getitem?mr=#1}{#2}
}
\providecommand{\href}[2]{#2}

\end{document}